\documentclass[12pt]{amsart}
\usepackage[margin=1.1in]{geometry}
\usepackage{amssymb,amsfonts,amsmath,mathrsfs,bbm}
\usepackage[unicode]{hyperref}
\usepackage[capitalise]{cleveref}
\usepackage{constants}

\hypersetup{colorlinks=true, citecolor=blue, linkcolor=blue, urlcolor=blue, pdfstartview=FitH, pdfauthor=Gergely Harcos and Jesse Thorner, pdftitle=A new zero-free region for Rankin--Selberg L-functions}

\newtheorem{theorem}{Theorem}[section]
\newtheorem{corollary}[theorem]{Corollary}

\newtheorem{proposition}[theorem]{Proposition}
\newtheorem{lemma}[theorem]{Lemma}

\newtheorem{question*}{Question}
\newtheorem{problem*}{Problem}
\newtheorem*{hypothesis*}{Hypothesis}
\newtheorem*{definition*}{Definition}

\theoremstyle{definition}

\theoremstyle{remark}
\newtheorem*{remark}{Remark}

\numberwithin{equation}{section}

\crefname{figure}{Figure}{Figures}
\theoremstyle{plain}
\newtheorem*{theorem*}{Theorem}
\newtheorem*{lemma*}{Lemma}
\crefname{theorems}{Theorem}{Theorems}
\crefname{corollaries}{Corollary}{Corollaries}
\newtheorem*{corollary*}{Corollary}
\crefname{corollaries*}{Corollary}{Corollaries}
\crefname{lemma}{Lemma}{Lemmata}
\crefname{proposition}{Proposition}{Propositions}
\crefname{conjectures}{Conjecture}{Conjectures}
\newtheorem*{conjonjecture*}{Conjecture}
\crefname{conjonjectures*}{Conjecture}{Conjectures}
\crefname{definitions}{Definition}{Definitions}
\crefname{hypotheses}{Hypothesis}{Hypotheses}

\allowdisplaybreaks

\renewcommand{\hat}{\widehat}
\renewcommand{\tilde}{\widetilde}
\renewcommand{\bar}{\overline}
\renewcommand{\epsilon}{\varepsilon}
\renewcommand{\pmod}[1]{\,(\mathrm{mod}\,\,#1)}
\renewcommand{\Re}{\mathrm{Re}}
\renewcommand{\Im}{\mathrm{Im}}
\newcommand{\re}{\Re}
\newcommand{\im}{\Im}
\newcommand{\GL}{\mathrm{GL}}
\newcommand{\N}{\mathrm{N}}
\newcommand{\Sym}{\mathrm{Sym}}
\newcommand{\A}{\mathbb{A}}

\newcommand{\CC}{\mathbb{C}}
\newcommand{\R}{\mathbb{R}}
\newcommand{\Q}{\mathbb{Q}}

\newcommand{\ka}{\mathfrak{a}}
\newcommand{\kp}{\mathfrak{p}}
\newcommand{\kq}{\mathfrak{q}}
\newcommand{\cL}{\mathcal{L}}
\newcommand{\cO}{\mathcal{O}}

\newconstantfamily{abcon}{symbol=c}

\title[A new zero-free region for Rankin--Selberg $L$-functions]{A new zero-free region for Rankin--Selberg $L$-functions}

\author{Gergely Harcos}
\address{Alfr{\'e}d R{\'e}nyi Institute of Mathematics, POB 127, Budapest H-1364, Hungary}
\email{\href{mailto:gharcos@renyi.hu}{gharcos@renyi.hu}}

\author{Jesse Thorner}
\address{Department of Mathematics, University of Illinois, Urbana, IL 61801, USA}
\email{\href{mailto:jesse.thorner@gmail.com}{jesse.thorner@gmail.com}}

\begin{document}

\begin{abstract}
Let $\pi$ and $\pi'$ be cuspidal automorphic representations of $\mathrm{GL}(n)$ and $\mathrm{GL}(n')$ with unitary central characters. We establish a new zero-free region for all $\mathrm{GL}(1)$-twists of the Rankin--Selberg $L$-function $L(s,\pi\times\pi')$, generalizing Siegel's celebrated work on Dirichlet $L$-functions. As an application, we prove the first unconditional Siegel--Walfisz theorem for the Dirichlet coefficients of $-L'(s,\pi\times\pi')/L(s,\pi\times\pi')$. Also, for $n\leq 8$, we extend the region of holomorphy and nonvanishing for the twisted symmetric power $L$-functions $L(s,\pi,\Sym^n\otimes\chi)$ of any cuspidal automorphic representation of $\mathrm{GL}(2)$.
\end{abstract}

\thanks{The first author was supported by the MTA–HUN-REN RI Lend\"ulet Automorphic Research Group and NKFIH (National Research, Development and Innovation Office) grant K~143876.  The second author was partially supported by the National Science Foundation (DMS-2401311) and the Simons Foundation (MP-TSM-00002484).}


\maketitle

\section{Introduction and the main result}
\label{sec:intro}

In 1896, Hadamard and de la Vall{\'e}e Poussin independently proved that the Riemann zeta function $\zeta(s)$ does not vanish in the half-plane $\Re(s)\geq 1$. This statement is equivalent to the asymptotic form of the prime number theorem. In 1899, de la Vall{\'e}e Poussin established the classical zero-free region for $\zeta(s)$, which was quickly extended to Dirichlet $L$-functions. In particular, let $\chi\pmod{q}$ be a Dirichlet character. There exists an absolute and effectively computable constant $\Cl[abcon]{ZetaZFR}>0$ such that $L(s,\chi)$ has at most one zero $\beta$ (necessarily real and simple) in the region
\begin{equation}
\label{eqn:standard_Dirichlet}
\Re(s)\geq 1-\Cr{ZetaZFR}/\log(q(|\im(s)|+3)).
\end{equation}
If $\beta$ exists, then $\chi$ is quadratic. In this case, Siegel's lower bound on $L(1,\chi)$ implies that for any $\epsilon>0$, there exists an ineffective constant $\Cl[abcon]{SiegelDirichlet}=\Cr{SiegelDirichlet}(\epsilon)>0$ such that
\begin{equation}
\label{eqn:Siegel_Dirichlet}
L(\sigma,\chi)\neq 0,\qquad \sigma\geq 1-\Cr{SiegelDirichlet}q^{-\epsilon}.	
\end{equation}
See \cite{Siegel,Walfisz} for the original references.

The method of de la Vall{\'e}e Poussin can be modified to establish a zero-free region for many $L$-functions. Specifically, let $F$ be a number field, let $\mathbb{A}_F$ be the ring of adeles over $F$, and let $\mathfrak{F}_{n}$ be the set of cuspidal automorphic representations $\pi$ of $\GL_{n}(\mathbb{A}_F)$ whose central character $\omega_{\pi}$ is unitary. Given $\pi\in\mathfrak{F}_{n}$, let $\tilde{\pi}\in\mathfrak{F}_{n}$ be the contragredient and $L(s,\pi)$ be the $L$-function of $\pi$. Let $C(\pi)\geq 3$ denote the analytic conductor, which measures the arithmetic and spectral complexity of $\pi$. Given $(\pi,\chi)\in\mathfrak{F}_{n}\times\mathfrak{F}_{1}$, let $\pi\otimes\chi$ be the cuspidal automorphic representation $g\mapsto\pi(g)\chi(\det g)$. For convenience, we introduce the subset $\mathfrak{F}_n^*$ consisting of $\pi\in \mathfrak{F}_n$ for which $\omega_\pi$ is trivial on the diagonally embedded positive reals. For each $\pi\in\mathfrak{F}_n$, there exist unique $\pi^*\in\mathfrak{F}_n^*$ and $t_{\pi}\in\R$ such that
\[
\pi=\pi^*\otimes|\cdot|^{it_{\pi}},\qquad L(s,\pi)=L(s+it_{\pi},\pi^*).
\]

Given $(\pi,\pi')\in\mathfrak{F}_n\times\mathfrak{F}_{n'}$, let $L(s,\pi\times\pi')$ be the associated Rankin--Selberg $L$-function, whose basic properties were established by Jacquet, Piatetski-Shapiro, and Shalika~\cite{JPSS,JS1,JS2}. If $(\pi,\pi')\in\mathfrak{F}_n^*\times\mathfrak{F}_{n'}^*$, then $L(s,\pi\times\pi')$ is holomorphic away from a possible pole at $s=1$, which occurs if and only if $\pi'=\tilde{\pi}$. If $\pi'\in\mathfrak{F}_1^*$ is trivial, then $L(s,\pi\times\pi')=L(s,\pi)$.

In 1976, Jacquet and Shalika~\cite[Theorem~1.3]{JacquetShalika} proved that $L(s,\pi)$ does not vanish in the half-plane $\Re(s)\geq 1$. In 1981, Shahidi~\cite[Theorem~5.2]{Shahidi} established the same zero-free region for $L(s,\pi\times\pi')$. Replacing $\pi$ with $\pi\otimes|\cdot|^{it}$ and varying $t\in\R$, we find that these results are equivalent to proving that for all $(\pi,\pi')\in\mathfrak{F}_n\times\mathfrak{F}_{n'}$, we have that $L(\sigma,\pi)\neq 0$ and $L(\sigma,\pi\times\pi')\neq 0$ when $\sigma\geq 1$.

In some cases, the method of de la Vall{\'e}e Poussin can extend this zero-free region.\footnote{In \cite{Brumley,GoldfeldLi,Humphries,HumphriesThorner2,HumphriesThorner,JLTW,Li,Molteni,Moreno,SoundararajanThorner}, it is assumed that $(\pi,\pi')\in\mathfrak{F}_n^*\times\mathfrak{F}_{n'}^*$, so $t_{\pi}=t_{\pi'}=0$.} Brumley~\cite[Theorem~A.1]{Humphries} and Humphries--Thorner~\cite[Theorem~2.1]{HumphriesThorner} (see also Moreno~\cite{Moreno}) proved that there exists an effectively computable constant $\Cl[abcon]{ZFR2_std}=\Cr{ZFR2_std}(n,n',[F:\Q])>0$ with the following property. If $(\pi,\pi')\in\mathfrak{F}_n\times\mathfrak{F}_{n'}$,
\begin{equation}
\label{eqn:special_star}
\pi=\tilde\pi\qquad\textup{or}
\qquad\pi'=\tilde\pi'\qquad\textup{or}
\qquad\pi'^*=\tilde\pi^*,
\end{equation}
and $t_{\pi}+t_{\pi'}\neq 0$, then $L(\sigma,\pi\times\pi')\neq 0$ in the interval
\begin{equation}
\label{eqn:standard_ZFR}
\sigma\geq 1-\Cr{ZFR2_std}/\log(C(\pi)C(\pi')).
\end{equation}
If \eqref{eqn:special_star} holds and $t_{\pi}+t_{\pi'}=0$, then $L(\sigma,\pi\times\pi')$ has at most one zero $\beta_{\pi\times\pi'}$ in the interval \eqref{eqn:standard_ZFR}.  If $\beta_{\pi\times\pi'}$ exists, then it is simple, and
\begin{equation}
\label{eqn:SZself-dual}
(\pi,\pi')=(\tilde{\pi},\tilde{\pi}')\qquad \textup{or}\qquad\pi'=\tilde\pi.
\end{equation}
When $\pi'\in\mathfrak{F}_1$ is trivial, we recover the standard zero-free region for $L(s,\pi)$.

There are limited (though important) cases where \eqref{eqn:SZself-dual} holds and exceptional zeros have been precluded altogether. For example, if $\pi'\in\mathfrak{F}_1$ is trivial and $\pi\in\mathfrak{F}_2\cup\mathfrak{F}_3$, then the exceptional zero does not exist \cite{Banks,HoffsteinLockhart,HoffsteinRamakrishnan}. See \cite{Luo,RamakrishnanWang} for examples when $\pi'\in\mathfrak{F}_2$ and $\pi\in\mathfrak{F}_2\cup\mathfrak{F}_3$. Generalizing \eqref{eqn:Siegel_Dirichlet}, Jiang--L{\"u}--Thorner--Wang~\cite[Section~4]{JLTW} and Humphries--Thorner~\cite[Theorem~2.4]{HumphriesThorner2} proved that if $\pi\in\mathfrak{F}_n$ and $\chi\in \mathfrak{F}_1$ is quadratic (in which case $t_{\chi}=0$), then for all $\epsilon>0$, there exists an ineffective constant $\Cl[abcon]{Siegel_pre}=\Cr{Siegel_pre}(\pi,\epsilon)>0$ such that
\[
L(\sigma,\pi\otimes\chi)\neq 0,\qquad L(\sigma,\pi\otimes(\tilde{\pi}\otimes\chi))\neq 0,\qquad \sigma\geq 1-\Cr{Siegel_pre}C(\chi)^{-\epsilon}.
\]
See also Molteni~\cite{Molteni}.

If $(\pi,\pi')\in\mathfrak{F}_n\times\mathfrak{F}_{n'}$ does not satisfy \eqref{eqn:special_star}, then in general, it is unclear how to execute the method of de la Vall{\'e}e Poussin. For such pairs, Brumley (\cite{Brumley}, \cite[Theorem~A.1]{Lapid}) made the first uniform improvement over \cite[Theorem~5.2]{Shahidi}, proving that for all $\epsilon>0$, there exists an effectively computable constant $\Cl[abcon]{ZFR_Brumley}=\Cr{ZFR_Brumley}(n,n',F,\epsilon)>0$ such that $L(\sigma,\pi\times\pi')\neq 0$ when
\begin{equation}
\label{eqn:Brumley_ZFR}
\sigma\geq 1-\Cr{ZFR_Brumley}/(C(\pi)C(\pi'))^{n+n'-1+\epsilon}.
\end{equation}
See Zhang~\cite{Zhang} for a recent numerical improvement in the exponent of  \eqref{eqn:Brumley_ZFR}.

Many others have established zero-free regions for $L$-functions without the method of de la Vall{\'e}e Poussin. For example, Sarnak~\cite{Sarnak} and Gelbart--Lapid~\cite{GelbartLapid} established zero-free regions using Eisenstein series and the Maa{\ss}--Selberg relations. See also the work of Goldfeld--Li~\cite[Theorem~1.2]{GoldfeldLi} and Humphries~\cite[Theorem~1.9]{Humphries} on $L(s,\pi\times\tilde{\pi})$.

In this paper, we develop a new method for establishing zero-free regions for $L$-functions. We apply our method to extend Siegel's celebrated result \eqref{eqn:Siegel_Dirichlet} to every $\GL_1$-twist of $L(s,\pi\times\pi')$, where $(\pi,\pi')\in\mathfrak{F}_n\times\mathfrak{F}_{n'}$. Our proof relies crucially on the group structure of $\mathfrak{F}_1$. We substantially improve \eqref{eqn:Brumley_ZFR} in the $\mathrm{GL}_1$-twist aspect, but the dependence on $\pi$ and $\pi'$ is no longer effective.

\begin{theorem}
\label{thm:Siegel}
Let $\pi\in\mathfrak{F}_n$ and $\pi'\in\mathfrak{F}_{n'}$. For all $\epsilon>0$, there exists an ineffective constant $\Cl[abcon]{ZFR2}=\Cr{ZFR2}(\pi,\pi',\epsilon)>0$ such that if $\chi\in\mathfrak{F}_1$, then
\begin{equation}
\label{eqn:finalbound2}
|L(\sigma,\pi\times(\pi'\otimes\chi))|\geq\Cr{ZFR2}C(\chi)^{-\epsilon},\qquad \sigma\geq 1-\Cr{ZFR2}C(\chi)^{-\epsilon}.
\end{equation}
\end{theorem}

\begin{remark}
By replacing $\chi$ with $\chi|\cdot|^{it}$, the bound \eqref{eqn:finalbound2} becomes
\begin{equation}
\label{eqn:ZFR}
|L(\sigma+it,\pi\times(\pi'\otimes\chi))|\geq\Cr{ZFR2}C(it,\chi)^{-\epsilon},\qquad \sigma\geq 1-\Cr{ZFR2}C(it,\chi)^{-\epsilon},
\end{equation}
where $C(it,\chi)\ll C(\chi)(|t|+1)^{[F:\Q]}$. In particular, there exists an ineffective constant $\Cl[abcon]{C777}=\Cr{C777}(\pi,\pi',\epsilon)>0$ such that
\[|L(\sigma+it,\pi\times\pi')|\geq\Cr{C777}(|t|+1)^{-\epsilon},\qquad \sigma\geq 1-\Cr{C777}(|t|+1)^{-\epsilon}.	\]
This $t$-aspect lower bound can be used to bound Eisenstein series $E(g,\varphi,s)$ on $\GL_{n+n'}$ coming from cusp forms $\varphi$ on the Levi factor $\GL_{n}\times\GL_{n'}$ of $\GL_{n+n'}$. See \cite[Corollary~2]{GelbartLapid} and its proof for details.
\end{remark}
\begin{remark}
One might hope to prove a version of \cref{thm:Siegel}, reminiscent of Tatuzawa's refinement of Siegel's theorem \cite{Tatuzawa}. The goal would be to prove that there exists $\psi\in\mathfrak{F}_1$ depending on $(\pi,\pi',\epsilon)$ such that, for every $\chi\in\mathfrak{F}_1$, the bound \eqref{eqn:finalbound2} holds with an \emph{effective} constant $c_6$ as long as $L(s,\pi\times(\pi'\otimes\chi))$ is not a shift of $L(s,\pi\times(\pi'\otimes\psi))$. Ichihara and Matsumoto~\cite{IM} studied such Tatuzawa-type estimates for general $L$-functions. However, their hypotheses are very restrictive, and they do not seem applicable in the setting of \cref{thm:Siegel}.
\end{remark}

After describing two applications in \cref{sec:Applications} and reviewing standard properties of $L$-functions in \cref{sec:Properties}, we provide an overview of the strategy of our proof in \cref{sec:Strategy}.  We execute our strategy in Sections~\ref{sec:proof_prop}--\ref{sec:finish} and deduce our applications in Sections~\ref{sec:PNTAP}--\ref{sec:sym}.

\subsection*{Acknowledgements}
We thank Erez Lapid, Paul Nelson, and the anonymous referee for their helpful remarks.

\section{Two applications}
\label{sec:Applications}

\subsection{Generalizing the Siegel--Walfisz theorem}
\label{subsec:application}
Let $\cO_F$ be the ring of integers of $F$, and define the norm $\N=\N_{F/\Q}$ on the nonzero ideals $\ka$ of $\cO_F$ by $\N\ka=|\cO_F/\ka|$. Let $\kq$ be a nonzero ideal of $\cO_F$, and let $I(\kq)$ be the group of fractional ideals that are coprime to $\kq$. Let $P(\kq)$ be the subgroup of $I(\kq)$ consisting of principal fractional ideals $(\alpha)$ such that $\alpha$ is totally positive and $\alpha\equiv 1\pmod{\kq}$. The narrow ray class group $\mathrm{Cl}(\kq)$ is the finite abelian quotient $I(\kq)/P(\kq)$. Let $\mathscr{P}(\kq)$ be the set of primitive ray class characters that induce the characters of $\mathrm{Cl}(\kq)$.

If there exists $u\in\R$ such that $(\pi,\pi')\in\mathfrak{F}_n\times\mathfrak{F}_{n'}$ satisfies $\pi'=\tilde{\pi}\otimes|\cdot|^{iu}$, then we define $\mathcal{M}_{\pi\times\pi'}(x)=x^{1-iu}/(1-iu)$. For all other pairs $(\pi,\pi')\in\mathfrak{F}_n\times\mathfrak{F}_{n'}$, we define $\mathcal{M}_{\pi\times\pi'}(x)=0$. We also define $\Lambda_{\pi\times\pi'}(\ka)$ by the Dirichlet series identity
\begin{equation}
\label{eqn:Lambda_def}
\sum_{\ka}\frac{\Lambda_{\pi\times\pi'}(\ka)}{\N\ka^s}=-\frac{L'}{L}(s,\pi\times\pi'),\qquad\Re(s)>1.
\end{equation}
For a class $\mathcal{C}\in \mathrm{Cl}(\kq)$ and $x\geq 2$, we define
\begin{equation}
\label{eqn:Edef}
\mathcal{E}_{\pi\times\pi'}(x;\kq,\mathcal{C})=\sum_{\substack{\N\ka\leq x \\ \ka\in\mathcal{C}}}\Lambda_{\pi\times\pi'}(\ka)-\frac{1}{|\mathrm{Cl}(\kq)|}\sum_{\chi\in\mathscr{P}(\kq)}\overline{\chi}(\mathcal{C})\mathcal{M}_{\pi\times(\pi'\otimes\chi)}(x).
\end{equation}
The bound $\mathcal{E}_{\pi\times\pi'}(x;\kq,\mathcal{C})=o(x)$ is equivalent to the fact that if $\chi\in\mathscr{P}(\kq)$ and $\Re(s)\geq 1$, then $L(s,\pi\times(\pi'\otimes\chi))\neq 0$.

Unconditionally, Brumley's narrow zero-free region \eqref{eqn:Brumley_ZFR} suffices to prove that there exists an effectively computable constant $\Cl[abcon]{Brumley_PNTAP1}=\Cr{Brumley_PNTAP1}(n,n',[F:\Q])>0$ such that if $\N\kq\leq(\log x)^{\Cr{Brumley_PNTAP1}}$, then $\mathcal{E}_{\pi\times\pi'}(x;\kq,\mathcal{C})\ll_{\pi,\pi'}x(\log x)^{-\Cr{Brumley_PNTAP1}}$. Before the present work, the bound $\mathcal{E}_{\pi\times\pi'}(x;\kq,\mathcal{C})\ll_{\pi,\pi',A}x(\log x)^{-A}$ for all $A>0$ was only known in limited cases, as a consequence of the zero-free regions proved in \cite{Humphries,HumphriesThorner2,HumphriesThorner,JLTW}. We prove the following result using \cref{thm:Siegel}.
\begin{theorem}
\label{thm:PNTAP}
Let $\pi\in\mathfrak{F}_n$, $\pi'\in\mathfrak{F}_{n'}$, $x\geq 2$, and $A>0$. For a nonzero ideal $\kq$ of $\cO_F$, let $\mathrm{Cl}(\kq)$ be the associated ray class group, and let $\mathcal{C}\in\mathrm{Cl}(\kq)$. If $\mathcal{E}_{\pi\times\pi'}(x;\kq,\mathcal{C})$ is given by \eqref{eqn:Edef} and $\N\kq\leq(\log x)^A$, then
\[
\mathcal{E}_{\pi\times\pi'}(x;\kq,\mathcal{C})\ll_{\pi,\pi',A}x(\log x)^{-A}.
\]
The implied constant is ineffective.
\end{theorem}

\begin{remark}
Let $a,q\geq 1$ be integers such that $\gcd(a,q)=1$, and let $\varphi(q)$ be Euler's totient function. If $F=\Q$ and $q$ is coprime to the conductors of $\pi$ and $\pi'$, then
\[
\sum_{\substack{m\leq x \\ m\equiv a\pmod{q}}}\Lambda_{\pi\times\pi'}(m)=\frac{\mathcal{M}_{\pi\times\pi'}(x)}{\varphi(q)}+O_{\pi,\pi',A}\left(\frac{x}{(\log x)^A}\right),	\qquad q\leq (\log x)^A
\]
by \cref{thm:PNTAP}.  This generalizes the Siegel--Walfisz theorem \cite{Siegel,Walfisz}.
\end{remark}

\subsection{Holomorphy and nonvanishing of symmetric power $L$-functions}

Let $\pi\in\mathfrak{F}_2$. Let $v$ be a place of $F$, and let $F_v$ be the completion of $F$ relative to $v$. We express $\pi$ as a restricted tensor product $\bigotimes_v\pi_v$ of smooth, admissible representations of $\GL_2(F_v)$. If $n\geq 2$, then $\Sym^n\colon \GL_2(\mathbb{C})\to\GL_{n+1}(\mathbb{C})$ is the $(n+1)$-dimensional irreducible representation of $\GL_2(\mathbb{C})$ on symmetric tensors of rank $n$. Let $\varphi_v$ be the two-dimensional representation of the Deligne--Weil group attached to $\pi_v$ and $\Sym^n(\pi_v)$ be the smooth admissible representation of $\GL_{n+1}(F_v)$ attached to the representation $\Sym^n\circ\varphi_v$. By the local Langlands correspondence, $\Sym^n(\pi_v)$ is well-defined for every place $v$ of $F$. Writing
\[
L(s,\pi) = \prod_v L(s,\pi_v) = \prod_{\kp}\prod_{j=1}^2\frac{1}{1-\alpha_{j,\pi}(\kp)\N\kp^{-s}}
\]
for the $L$-function of $\pi$, we define
\[
L(s,\pi,\Sym^n) = \prod_v L(s,\Sym^n(\pi_v)) \doteq 
\prod_{\kp\nmid\kq_{\pi}} \prod_{m=0}^n\frac{1}{1-\alpha_{1,\pi}(\kp)^{m}\alpha_{2,\pi}(\kp)^{n-m}\N\kp^{-s}}.
\]
Here, $\doteq$ means that we have suppressed the more complicated Euler factors at the prime ideals $\kp$ dividing the conductor $\kq_{\pi}$ of $\pi$.

Langlands conjectured that $\Sym^n(\pi) = \bigotimes_v \Sym^n(\pi_v)$ is an automorphic representation of $\GL_{n+1}(\A_F)$. In settings where the conjecture is proved, we write $L(s,\Sym^n(\pi))$ instead of $L(s,\pi,\Sym^n)$. As of now, the conjecture is proved for $n\leq 4$ \cite{GJ,Kim,KimShahidi}, in which case there exists a least integer $r_{n,\pi}\geq 0$ such that $(s+int_{\pi}-1)^{r_{n,\pi}}L(s,\Sym^n(\pi))$ is entire. Moreover, there exists an effectively computable constant $\Cl[abcon]{symm_std}=\Cr{symm_std}(n,[F:\Q])>0$ such that $L(s,\Sym^n(\pi))$ has at most one zero $\rho=\beta-int_{\pi}$, necessarily simple with real part $\beta<1$, in the region
\begin{equation}
\label{eqn:Symm}
\Re(s)\geq 1-\Cr{symm_std}/\log(C(\pi^*)(|\im(s)+nt_{\pi}|+3)).
\end{equation}

Kim and Shahidi~\cite{KimShahidi2} used the results of \cite{GJ,Kim,KimShahidi}, Rankin--Selberg theory, and the Clebsch--Gordan identities to prove that if $5\leq n\leq 8$, then there exists a least integer $r_{n,\pi}\geq 0$ such that $(s+int_{\pi}-1)^{r_{n,\pi}}L(s,\pi,\Sym^n)$ is holomorphic and nonvanishing in the half-plane $\Re(s)\geq 1$. The region of holomorphy and nonvanishing can be widened inside the critical strip using Brumley's narrow zero-free region \eqref{eqn:Brumley_ZFR}, and the region will depend effectively on $\pi$.

Using \cref{thm:Siegel}, we establish an even larger region of holomorphy and nonvanishing, albeit with ineffective dependence on $\pi$. We also accommodate $\GL_1$-twists. In keeping with the notation in \cite{KimShahidi2}, if $\chi\in\mathfrak{F}_1$, then we write
\begin{align*}
L(s,\pi,\Sym^n\otimes\chi) &= \prod_v L(s,\Sym^n(\pi_v)\otimes\chi_v)\\
&\doteq \prod_{\kp\nmid\kq_{\pi}}\prod_{m=0}^{n}\frac{1}{1-\alpha_{1,\pi}(\kp)^{m}\alpha_{2,\pi}(\kp)^{n-m}\chi(\kp)\N\kp^{-s}}
\end{align*}
for the twisted $n$-th symmetric power $L$-function of $\pi$. If $\Sym^n(\pi)$ is an automorphic representation of $\GL_{n+1}(\A_F)$, then we write $L(s,\Sym^n(\pi)\otimes\chi)$ instead.

\begin{theorem}
\label{thm:sym}
Let $(\pi,\chi)\in\mathfrak{F}_2\times\mathfrak{F}_1$. Let $n\leq 8$, and let $r_{n,\pi,\chi}\geq 0$ be the least integer such that the function
\begin{equation}
\label{cL1}
\cL(s,\pi,\Sym^n\otimes\chi)=\left(\frac{s+int_{\pi}+it_{\chi}-1}{s+int_{\pi}+it_{\chi}+1}\right)^{r_{n,\pi,\chi}}
L(s,\pi,\Sym^n\otimes\chi)
\end{equation}
holomorphically continues to $\Re(s)\geq 1$. For all $\epsilon>0$, there exists an ineffective constant $\Cl[abcon]{Sym}=\Cr{Sym}(\pi,\epsilon)\geq 1$ such that if $\sigma\geq 1-\Cr{Sym}^{-1}C(\chi)^{-\epsilon}$, then $\cL(s,\pi,\Sym^n\otimes\chi)$ is holomorphic at $s=\sigma$ and satisfies
\[
\Cr{Sym}^{-1}C(\chi)^{-\epsilon}\leq|\cL(\sigma,\pi,\Sym^n\otimes\chi)|\leq \Cr{Sym}C(\chi)^{\epsilon}.
\]
\end{theorem}

This is the strongest known region of holomorphy and nonvanishing for general fixed $\pi\in\mathfrak{F}_2$. Several remarks are in order. First, in \cref{thm:sym}, both the lower bound and the upper bound rely on \cref{thm:Siegel}. In particular, one cannot deduce the upper bound solely from convexity (see \cref{lem:Li1}). Second, for certain $\pi\in\mathfrak{F}_2$, much more is known. For example, if $F$ is totally real and $\pi\in\mathfrak{F}_2$ corresponds with a primitive holomorphic non-CM Hilbert cusp form over $F$ of weights $k_v\geq 2$ for $v\mid\infty$, then $\Sym^n(\pi)\in\mathfrak{F}_{n+1}$ holds for all $n\geq 1$ by the recent work of Newton and Thorne~\cite{NewtonThorne,NewtonThorne2,NewtonThorne3}. Hence, if $\chi\in\mathfrak{F}_1$, then $L(s,\Sym^n(\pi)\otimes\chi)$ is entire and has the zero-free region \eqref{eqn:Symm}. Third, as a corollary of the aforementioned work of Gelbart and Lapid~\cite[Corollary~2]{GelbartLapid}, there exists an effectively computable constant $\Cl[abcon]{Sym9}=\Cr{Sym9}(\pi)>0$ such that if $t\in\R$, then
\[
|L(1+it,\pi,\Sym^9)|\geq (|t|+3)^{-\Cr{Sym9}}.
\]
Our methods cannot say anything about the holomorphy or nonvanishing of $L(s,\pi,\Sym^9)$ without first knowing that $\Sym^5(\pi)$ is automorphic.

\section{Properties of $L$-functions}
\label{sec:Properties}

Let $F$ be a number field with adele ring $\A_F$. Let $D_F$ be the absolute discriminant of $F$, $\cO_F$ the ring of integers of $F$, and $\N=\N_{F/\Q}$ the norm defined on nonzero ideals $\ka$ of $\cO_F$ by $\N\ka=|\cO_F/\ka|$. For a place $v$ of $F$, let $v\mid\infty$ (resp. $v\nmid\infty$) denote that $v$ is archimedean (resp. non-archimedean), and let $F_v$ be the corresponding completion of $F$. Each $v\nmid\infty$ corresponds with a prime ideal $\kp$ of $\cO_F$.

\subsection{Standard $L$-functions}
\label{subsec:standard}

Let $\mathfrak{F}_n$ be the set of cuspidal automorphic representations $\pi$ of $\GL_n(\A_F)$ whose central character $\omega_\pi$ is unitary, and let $\mathfrak{F}_n^*$ be the subset of those $\pi$'s for which $\omega_\pi$ is trivial on the diagonally embedded positive reals. Every $\pi\in\mathfrak{F}_n$ can be written uniquely as
\begin{equation}
\label{eqn:pidecomp}
\pi=\pi^*\otimes|\cdot|^{it_{\pi}},\qquad \pi^*\in\mathfrak{F}_n^*,\quad t_{\pi}\in\R.
\end{equation}

If $\pi\in\mathfrak{F}_n^*$, then for each place $v$, there exists an irreducible admissible representation $\pi_v$ of $\GL_n(F_v)$, with $\pi_v$ ramified for at most finitely many $v$, such that $\pi$ is a restricted tensor product $\otimes_v \pi_v$. When $v\nmid\infty$ and $\kp$ corresponds with $v$, then we write $\pi_v$ and $\pi_{\kp}$ interchangeably. For each prime ideal $\kp$, there exist $n$ Satake parameters $\alpha_{1,\pi}(\kp),\dotsc,\alpha_{n,\pi}(\kp)\in\mathbb{C}$ such that the standard $L$-function $L(s,\pi)$ of $\pi$ is the absolutely convergent Euler product
\[
L(s,\pi)=\prod_{\kp}L(s,\pi_{\kp}) = \prod_{\kp}\prod_{j=1}^n \frac{1}{1-\alpha_{j,\pi}(\kp)\N\kp^{-s}},\qquad\Re(s)>1.
\]
For $v\mid \infty$, define
\[
\Gamma_v(s)=\begin{cases}
\pi^{-s/2}\Gamma(s/2)&\mbox{if $F_v=\R$,}\\
2(2\pi)^{-s}\Gamma(s)&\mbox{if $F_v=\mathbb{C}$.}
\end{cases}
\]
There exist $n$ Langlands parameters $\mu_{1,\pi}(v),\dotsc,\mu_{n,\pi}(v)\in\mathbb{C}$ such that
\[
L(s,\pi_{v})=\prod_{j=1}^n \Gamma_v(s+\mu_{j,\pi}(v)).
\]

Let $\kq_{\pi}$ denote the conductor of $\pi$, and let $\mathbbm{1}\in\mathfrak{F}_1^*$ denote the trivial representation (whose $L$-function is the Dedekind zeta function $\zeta_F(s)$). Let $\delta_{\pi}=1$ if $\pi=\mathbbm{1}$, and $\delta_{\pi}=0$ otherwise. If
\[
L_{\infty}(s,\pi)=\prod_{v\mid\infty}L(s,\pi_v),
\]
then the completed $L$-function
\[
\Lambda(s,\pi)=(s(1-s))^{\delta_{\pi}}(D_F^n\N\kq_{\pi})^{s/2}L_{\infty}(s,\pi)L(s,\pi)
\]
is entire of order $1$. There exists $W(\pi)\in\CC$ of modulus $1$ such that $\Lambda(s,\pi)$ satisfies the functional equation
\[
\Lambda(s,\pi)=W(\pi)\Lambda(1-s,\tilde\pi)=W(\pi)\overline{\Lambda(1-\bar{s},\pi)}.
\]
The nontrivial zeros of $L(s,\pi)$ are the zeros of $\Lambda(s,\pi)$, and they lie in the critical strip $0<\Re(s)<1$. It is conjectured (GRH) that they actually lie on the critical line $\Re(s)=1/2$. Finally, the analytic conductor is $C(\pi)=C(0,\pi)$, where
\[
C(it,\pi) = D_F^n\N\kq_{\pi}\prod_{v\mid\infty}\prod_{j=1}^n (|\mu_{j,\pi}(v)+it|+3)^{[F_v:\R]}.
\]

More generally, these quantities are defined for all $\pi\in\mathfrak{F}_n$, and we have in particular
\[
L(s,\pi)=L(s+it_{\pi},\pi^*),\qquad C(it,\pi)=C(it+it_{\pi},\pi^*).
\]

\subsection{Rankin--Selberg $L$-functions}
\label{subsec:RS}

Let $\pi\in\mathfrak{F}_n^*$ and $\pi'\in\mathfrak{F}_{n'}^*$. For each $\kp\mid\kq_{\pi}\kq_{\pi'}$, there exist complex numbers $\alpha_{j,j',\pi\times\pi'}(\kp)$ with $1\leq j\leq n$ and $1\leq j'\leq n$ such that if
\[
L(s,\pi_{\kp}\times\pi_{\kp}')=\begin{cases}
\prod_{j=1}^n \prod_{j'=1}^{n'}(1-\alpha_{j,\pi}(\kp)\alpha_{j',\pi'}(\kp)\N\kp^{-s})^{-1}&\mbox{if $\kp\nmid\kq_{\pi}\kq_{\pi'}$,}\\
\prod_{j=1}^n \prod_{j'=1}^{n'}(1-\alpha_{j,j',\pi\times\pi'}(\kp)\N\kp^{-s})^{-1}&\mbox{if $\kp\mid\kq_{\pi}\kq_{\pi'}$,}
\end{cases}
\]
then the Rankin--Selberg $L$-function $L(s,\pi\times\pi')$ equals the absolutely convergent product
\begin{equation}
\label{eqn:euler_prod}
L(s,\pi\times\pi')=\prod_{\kp}L(s,\pi_{\kp}\times\pi_{\kp}')=\sum_{\ka}\frac{\lambda_{\pi\times\pi'}(\ka)}{\N\ka^s},\qquad \Re(s)>1.	
\end{equation}
By \cite[Lemma~3.1]{JLW}, we have the bound
\begin{equation}
\label{eqn:JLW_decouple}
|\lambda_{\pi\times\pi'}(\ka)|\leq \sqrt{\lambda_{\pi\times\tilde{\pi}}(\ka)\lambda_{\pi'\times\tilde{\pi}'}(\ka)}.	
\end{equation}

Let $\kq_{\pi\times\pi'}$ be the conductor of $L(s,\pi\times\pi')$. For each $v\mid\infty$, $1\leq j\leq n$, and $1\leq j'\leq n'$, there exists a Langlands parameter $\mu_{j,j',\pi\times\pi'}(v)$ such that if
\begin{equation}
\label{eqn:L1}
L_{\infty}(s,\pi\times\pi')=\prod_{v\mid\infty}\prod_{j=1}^n \prod_{j'=1}^{n'}\Gamma_v(s+\mu_{j,j',\pi\times\pi'}(v)),\quad \delta_{\pi\times\pi'}=\begin{cases}
	1&\mbox{if $\pi'=\tilde{\pi}$,}\\
	0&\mbox{otherwise,}
\end{cases}
\end{equation}
then the completed $L$-function
\begin{equation}
\label{eqn:L2}
\Lambda(s,\pi\times\pi')=(s(1-s))^{\delta_{\pi\times\pi'}}(D_F^{nn'}\N\kq_{\pi\times\pi'})^{s/2} L_{\infty}(s,\pi\times\pi')L(s,\pi\times\pi')
\end{equation}
is entire of order 1.  There exists $W(\pi\times\pi')\in\CC$ of modulus $1$ such that $\Lambda(s,\pi\times\pi')$ satisfies the functional equation
\begin{equation}
\label{eqn:FE}
\Lambda(s,\pi\times\pi')=W(\pi\times\pi')\Lambda(1-s,\tilde{\pi}\times\tilde{\pi}')=W(\pi\times\pi')\overline{\Lambda(1-\bar{s},\pi\times\pi')}.
\end{equation}

The absolute convergence of \eqref{eqn:euler_prod} ensures that $\re(\mu_{j,j',\pi\times\pi'}(v))\geq -1$. The nontrivial zeros of $L(s,\pi\times\pi')$ are the zeros of $\Lambda(s,\pi\times\pi')$, and they lie in the critical strip $0<\Re(s)<1$. It is conjectured (GRH) that they actually lie on the critical line $\Re(s)=1/2$. Finally, the analytic conductor is $C(\pi\times\pi')=C(0,\pi\times\pi')$, where
\begin{equation}
\label{eqn:C}
C(it,\pi\times\pi')=D_F^{nn'}\N\kq_{\pi\times\pi'}\prod_{v\mid\infty}\prod_{j=1}^n \prod_{j'=1}^{n'}
(|\mu_{j,j',\pi\times\pi'}(v)+it|+3)^{[F_v:\R]}.
\end{equation}
We have the bounds \cite[Lemma~A.2]{Humphries}
\begin{equation}
\label{eqn:BH}
C(it,\pi\times\pi')\ll_{n,n'} C(\pi\times\pi') (|t|+1)^{nn'[F:\Q]}\ll_{n,n'} C(\pi)^{n'}C(\pi')^n (|t|+1)^{nn'[F:\Q]}.
\end{equation}

More generally, for $\pi\in\mathfrak{F}_{n}$ and $\pi'\in\mathfrak{F}_{n'}$, we have
\[
L(s,\pi\times\pi') = L(s+it_{\pi}+it_{\pi'},\pi^*\times\pi'^*),\quad 
C(it,\pi\times\pi')=C(it+it_{\pi}+it_{\pi'},\pi^*\times\pi'^*).
\]

\subsection{Local bounds}

Let $\pi\in\mathfrak{F}_n$. At each prime ideal $\kp$ of $\cO_F$ and each $v\mid \infty$, there exists $\theta_n\in[0,1/2-1/(n^2+1)]$ such that \cite{LRS,MullerSpeh}
\begin{equation}
\label{eqn:GRC1}
|\alpha_{j,\pi}(\kp)|\leq \N\kp^{\theta_n},\qquad \re(\mu_{j,\pi}(v))\geq -\theta_n.
\end{equation}
Using the explicit descriptions of $\alpha_{j,j',\pi\times\pi'}(\kp)$ in \cite[Appendix]{SoundararajanThorner} and $\mu_{j,j',\pi\times\pi'}(v)$ in \cite[Section~3]{MullerSpeh}, we find that if $\pi\in\mathfrak{F}_n$ and $\pi'\in\mathfrak{F}_{n'}$, then
\begin{equation}
\label{eqn:GRC2}
|\alpha_{j,j',\pi\times\pi'}(\kp)|\leq \N\kp^{\theta_n+\theta_{n'}},\qquad \re(\mu_{j,j',\pi\times\pi'}(v))\geq -\theta_n-\theta_{n'}.
\end{equation}

\subsection{Isobaric sums}
\label{subsec:isobaric}

The Langlands theory of Eisenstein series associates to any $\ell$-tuple 
$(\pi_1,\dotsc,\pi_\ell)\in\mathfrak{F}_{n_1}\times\dotsb\times \mathfrak{F}_{n_\ell}$ an automorphic representation of $\GL_{n_1+\dotsb+n_\ell}(\A_F)$, the isobaric sum, denoted $\Pi=\pi_1\boxplus\dotsb\boxplus\pi_\ell$. The contragredient is $\tilde{\Pi}=\tilde{\pi}_1\boxplus\dotsb\boxplus\tilde{\pi}_\ell$, and
\[
L(s,\Pi)=\prod_{j=1}^\ell L(s,\pi_j),\qquad \Re(s)>1.
\]
Given two isobaric sums $\Pi=\pi_1\boxplus\dotsb\boxplus\pi_\ell$ and $\Pi'=\pi_1'\boxplus\dotsb\boxplus\pi_m'$, we define the Rankin--Selberg $L$-function
\[
L(s,\Pi\times\Pi')=\prod_{j=1}^\ell \prod_{k=1}^m L(s,\pi_j\times\pi_k')=\sum_{\ka}\frac{\lambda_{\Pi\times\Pi'}(\ka)}{\N\ka^s},\qquad \Re(s)>1.
\]
We let $\mathfrak{A}_n$ denote the set of isobaric automorphic representations of $\GL_{n}(\A_F)$.

\begin{lemma}[{\cite[Lemma~a]{HoffsteinRamakrishnan}}]
\label{lem:nonneg}
If $\Pi\in\mathfrak{A}_n$, then $L(s,\Pi\times\tilde{\Pi})$ has nonnegative Dirichlet coefficients, as do $\log L(s,\Pi\times\tilde{\Pi})$ and $-L'(s,\Pi\times\tilde{\Pi})/L(s,\Pi\times\tilde{\Pi})$.
\end{lemma}

\subsection{Convexity}

Our proofs require strong bounds for Rankin--Selberg $L$-functions and their derivatives. 

\begin{lemma}
\label{lem:Li1}
For $(\pi,\pi')\in\mathfrak{F}_n\times\mathfrak{F}_{n'}$, consider the holomorphic function
\begin{equation}
\label{cL2}
\cL(s,\pi\times\pi')=\left(\frac{s+it_{\pi}+it_{\pi'}-1}{s+it_{\pi}+it_{\pi'}+1}\right)^{\delta_{\pi^*\times\pi'^*}}L(s,\pi\times\pi'),\qquad\Re(s)>-1.
\end{equation}
If $j\geq 0$, $\sigma\geq 0$, and $\epsilon>0$, then
\begin{equation}
\label{eqn:Lbound1}
\cL^{(j)}(\sigma,\pi\times\pi')\ll_{n,n',[F:\Q],j,\epsilon}C(\pi\times\pi')^{\max(1-\sigma,0)/2+\epsilon}.
\end{equation}
In particular, if $\delta_{\pi^*\times\pi'^*}=0$ or $|t_{\pi}+t_{\pi'}|>1$, then
\[
L^{(j)}(\sigma,\pi\times\pi')\ll_{n,n',[F:\Q],j,\epsilon}C(\pi\times\pi')^{\max(1-\sigma,0)/2+\epsilon}.
\]
\end{lemma}

\begin{proof}
Using the notation $t=t_\pi+t_{\pi'}$, we can rewrite the bound \eqref{eqn:Lbound1} as
\begin{equation}
\label{eqn:Lbound2}\cL^{(j)}(\sigma+it,\pi^*\times\pi'^*)\ll_{n,n',[F:\Q],j,\epsilon}C(it,\pi^*\times\pi'^*)^{\max(1-\sigma,0)/2+\epsilon},
\end{equation}
where
\[
\cL(s,\pi^*\times\pi'^*)=\left(\frac{s-1}{s+1}\right)^{\delta_{\pi^*\times\pi'^*}}L(s,\pi^*\times\pi'^*),\qquad\Re(s)>-1.
\]
We shall prove \eqref{eqn:Lbound2} for any $t\in\R$, and at this point we relabel $(\pi^*,\pi'^*)$ as $(\pi,\pi')$ for notational simplicity. That is, we shall assume $(\pi,\pi')\in\mathfrak{F}_n^*\times\mathfrak{F}_{n'}^*$ for the rest of this proof.

Without loss of generality, $0<\epsilon<1$. It suffices to show for all $\sigma\geq -\epsilon/2$ that
\begin{equation}
\label{eqn:Lbound}
\cL(\sigma+it,\pi\times\pi')\ll_{n,n',[F:\Q],\epsilon}C(it,\pi\times\pi')^{\max(1-\sigma,0)/2+3\epsilon/4}.
\end{equation}
Indeed, assuming this statement and using the bound
\[
|\cL^{(j)}(s,\pi\times\pi')|\leq j!(2/\epsilon)^j\sup_{|z-s|=\epsilon/2}|\cL(z,\pi\times\pi')|
\]
that follows from Cauchy's formula and the triangle inequality for complex line integrals, we readily obtain the required conclusion (for all $j\geq 0$, $\sigma\geq 0$, and $t\in\R$).

It remains to prove \eqref{eqn:Lbound} for all $\sigma\geq -\epsilon/2$. Assume first that $\sigma\geq 1+\epsilon/2$. It follows from 
\eqref{eqn:JLW_decouple} and the Cauchy--Schwarz inequality that
\[
|L(\sigma+it,\pi\times\pi')|\leq(L(\sigma,\pi\times\tilde\pi)L(\sigma,\pi'\times\tilde\pi'))^{1/2}.
\]
The right-hand side is decreasing in $\sigma$; hence, by \cite[Theorem~2]{Li} we infer that
\[
\log|L(\sigma+it,\pi\times\pi')|\ll_{n,n',[F:\Q],\epsilon}
\frac{\log C(\pi\times\tilde\pi)}{\log\log C(\pi\times\tilde\pi)}+\frac{\log C(\pi'\times\tilde\pi')}{\log\log C(\pi'\times\tilde\pi')}.
\]
Then, using \cite[Lemma~2.1]{SoundararajanThorner}, we conclude that
\begin{equation}
\label{eqn:bound_on_1-line}
L(\sigma+it,\pi\times\pi')\ll_{n,n',[F:\Q],\epsilon}C(it,\pi\times\pi')^{\epsilon/4}.
\end{equation}
Concerning the application of \cite[Lemma~2.1]{SoundararajanThorner}, three remarks are in order. First, this result is meant for 
$F=\Q$, but a generalization to an arbitrary number field $F$ is straightforward (with extra constants depending on $[F:\Q]$ in the exponents). Second, \cite{SoundararajanThorner} tacitly assumes that each central character is trivial on the diagonally embedded positive reals, but this assumption is not used in the proof of \cite[Lemma~2.1]{SoundararajanThorner}. Third, we rely on the fact that $C(it,\pi\times\pi')$ in \eqref{eqn:bound_on_1-line} equals $C(\pi\times(\pi'\otimes|\cdot|^{it}))$.

Finally, we verify \eqref{eqn:Lbound} in the strip $-\epsilon/2\leq\sigma< 1+\epsilon/2$. We express the boundary value $L(-\epsilon/2+it,\pi\times\tilde\pi)$ from $L(1+\epsilon/2-it,\pi\times\tilde\pi)$ using the functional equation \eqref{eqn:FE} and the definitions \eqref{eqn:L1}--\eqref{eqn:L2}. Then, applying \cite[Lemma~3.2]{Harcos} and recalling \eqref{eqn:C}, we obtain from \eqref{eqn:bound_on_1-line} the bound
\begin{equation}
\label{eqn:bound_on_0-line}
L(-\epsilon/2+it,\pi\times\pi')\ll_{n,n',[F:\Q],\epsilon} C(it,\pi\times\pi')^{1/2+3\epsilon/4}.
\end{equation}
We interpolate between \eqref{eqn:bound_on_1-line} and \eqref{eqn:bound_on_0-line} using the Phragm{\'e}n--Lindel{\"o}f principle, and the lemma follows.
\end{proof}

\subsection{A weak Brun--Titchmarsh bound}

Let $\pi\in\mathfrak{F}_n$ and $\pi'\in\mathfrak{F}_{n'}$. Recall the Dirichlet series definition of $\Lambda_{\pi\times\pi'}(\ka)$ in \eqref{eqn:Lambda_def}. 

\begin{lemma}
\label{lem:BT}
Let $(\pi,\pi')\in\mathfrak{F}_{n}\times\mathfrak{F}_{n'}$. There exists an effectively computable constant $\Cl[abcon]{BTrange}=\Cr{BTrange}(n,n',[F:\Q])>0$ such that if
\[X\geq\max(C(\pi\times\tilde{\pi}),C(\pi'\times\tilde{\pi}'))^{\Cr{BTrange}}
\qquad\text{and}\qquad 1\leq T\leq X^{1/(16\max(n,n')^2[F:\Q])},
\]
then
\[
\sum_{X<\N\ka\leq Xe^{1/T}}|\Lambda_{\pi\times\pi'}(\ka)|\ll_{n,n',[F:\Q]} \frac{X}{T}.
\]
\end{lemma}

\begin{proof}
By \cite[Proposition~A.1]{SoundararajanThorner}, we have the bound
\begin{equation}
\label{eqn:Brumley_decouple}
|\Lambda_{\pi\times\pi'}(\ka)|\leq \sqrt{\Lambda_{\pi\times\tilde{\pi}}(\ka)\Lambda_{\pi'\times\tilde{\pi}'}(\ka)}.	
\end{equation}
Combined with the Cauchy--Schwarz inequality, \eqref{eqn:Brumley_decouple} yields
\begin{equation}
\label{eqn:BTCS}
\sum_{X<\N\ka\leq Xe^{1/T}}|\Lambda_{\pi\times\pi'}(\ka)|\leq \left(\sum_{X<\N\ka\leq Xe^{1/T}}\Lambda_{\pi\times\tilde{\pi}}(\ka)\right)^{1/2}\left(\sum_{X<\N\ka\leq Xe^{1/T}}\Lambda_{\pi'\times\tilde{\pi}'}(\ka)\right)^{1/2}.
\end{equation}

By \cite[Proposition~4.1]{HumphriesThorner}, there exist absolute and effectively computable constants $\Cl[abcon]{BT1}>0$ and $\Cl[abcon]{BT2}>0$ such that if $\log\log C(\pi\times\tilde{\pi})\geq \Cr{BT1}n^4[F:\Q]^2$, $X\geq e^{\Cr{BT2}n^4[F:\Q]^2}C(\pi\times\tilde{\pi})^{32n^2[F:\Q]}$, and $1\leq T\leq X^{1/(16n^2[F:\Q])}$, then
\[
\sum_{X<\N\ka\leq Xe^{1/T}}\Lambda_{\pi\times\tilde{\pi}}(\ka)\ll n^2[F:\Q]\frac{X}{T}.
\]
A careful inspection of the proof of \cite[Proposition~4.1]{HumphriesThorner} shows if the lower bound on $C(\pi\times\tilde{\pi})$ is removed, then the dependency on $n$ and $[F:\Q]$ in the implied constant and the range of $X$ become worse, which is not a concern here. We conclude that there exists an effectively computable constant $\Cr{BTrange}=\Cr{BTrange}(n,[F:\Q])>0$ such that if $X\geq C(\pi\times\tilde{\pi})^{\Cr{BTrange}}$ and $1\leq T\leq X^{1/(16n^2[F:\Q])}$, then
\[
\sum_{X<\N\ka\leq Xe^{1/T}}\Lambda_{\pi\times\tilde{\pi}}(\ka)\ll_{n,[F:\Q]} \frac{X}{T}.
\]
To finish, we similarly bound the second factor on the right-hand side of \eqref{eqn:BTCS}.
\end{proof}

\begin{corollary}
\label{cor:BT}
Let $(\pi,\pi',\chi)\in\mathfrak{F}_{n}\times\mathfrak{F}_{n'}\times\mathfrak{F}_1$. If
\[
X\geq 1\qquad\text{and}\qquad X\geq Y\geq X^{1-1/(16\max(n,n')^2[F:\Q])},
\]
then
\begin{equation}
\label{eqn:corBTbound}
\sum_{X<\N\ka\leq X+Y}|\Lambda_{\pi\times(\pi'\otimes\chi)}(\ka)|\ll_{\pi,\pi'}Y.
\end{equation}
\end{corollary}

\begin{proof}
We define $X_0=\max(C(\pi\times\tilde{\pi}),C(\pi'\times\tilde{\pi}'))^{\Cr{BTrange}}$. If $X\geq X_0$, then \eqref{eqn:corBTbound} follows from \cref{lem:BT} upon setting $T=X/Y$. If $X<X_0$, then \eqref{eqn:corBTbound} follows from \eqref{eqn:Brumley_decouple} and the Cauchy--Schwarz inequality.
\end{proof}

\section{Overall strategy}
\label{sec:Strategy}

Our proof of \cref{thm:Siegel} rests on the following key proposition.

\begin{proposition}
\label{prop:P1}
Let $(\pi,\pi',\chi)\in\mathfrak{F}_n\times\mathfrak{F}_{n'}\times\mathfrak{F}_1$, $\epsilon\in(0,1/2)$, and $\beta\in(1-\epsilon/8,1)$. 
Assume that the following $L$-functions are entire:
\begin{equation}
\label{eqn:Lfns}
L(s,\pi\times\pi'),\qquad L(s,\pi\times(\pi'\otimes\chi)),\qquad L(s,\pi\times(\pi'\otimes\chi^2)).
\end{equation}
If $L(\beta,\pi\times\pi')=0$, then
\begin{equation}
\label{eqn:main_lower}
|L(1,\pi\times(\pi'\otimes\chi))|\gg_{\pi,\pi',\beta,\epsilon}C(\chi)^{-(n+n')^2\epsilon}.
\end{equation}
\end{proposition}

Assuming \cref{prop:P1}, we will prove that, for any $\epsilon\in(0,1/2)$, there exists an ineffective constant $\Cl[abcon]{ZFR3}=\Cr{ZFR3}(\pi,\pi',\epsilon)>0$ such that for any $\chi\in\mathfrak{F}_1$ we have
\begin{equation}
\label{eqn:finalbound3}
|L(1,\pi\times(\pi'\otimes\chi))|\geq\Cr{ZFR3}C(\chi)^{-\epsilon}.
\end{equation}
After this, \cref{thm:Siegel} follows in a routine way (see \cref{sec:finish}). Let us define
\begin{equation}
\label{eqn:epsprime}
\epsilon'=\frac{\epsilon}{8(n+n')^2}.
\end{equation}
By a standard argument (see \cref{subsec:inital_reduction}), we only need to prove \eqref{eqn:finalbound3} 
in the case when $L(s,\pi\times(\pi'\otimes\chi))$ has a zero in the half-plane $\Re(s)>1-\epsilon'$.

It is instructive to first consider the case $n\neq n'$ of \eqref{eqn:finalbound3}. Indeed, in this case, the $L$-functions in \eqref{eqn:Lfns} are automatically entire. Moreover, by the ``existence of bad zero'' hypothesis, there exists $(\beta,\psi)\in(1-\epsilon',1)\times \mathfrak{F}_1$ depending at most on $(\pi,\pi',\epsilon)$ such that
\begin{equation}
\label{eqn:assumed_zero}
L(\beta,\pi\times(\pi'\otimes\psi))=0.
\end{equation}
Writing $\pi''=\pi'\otimes\psi\in\mathfrak{F}_{n'}$ and $\chi' = \bar{\psi}\chi\in\mathfrak{F}_1$, we apply \cref{prop:P1} with 
$(\pi,\pi'',\chi',8\epsilon')$ in the role of $(\pi,\pi',\chi,\epsilon)$. By $n\neq n'$ and \eqref{eqn:assumed_zero}, all the assumptions of \cref{prop:P1} are satisfied, hence we conclude that
\[
|L(1,\pi\times(\pi'\otimes\chi))|=|L(1,\pi\times(\pi''\otimes\chi'))|\gg_{\pi,\pi'',\beta,\epsilon}C(\chi')^{-\epsilon}.\]
This implies \eqref{eqn:finalbound3} readily, since $\pi''$ (resp. $\chi'$) differs from $\pi'$ (resp. $\chi$) by a twist of $\psi$, and both $\psi$ and $\beta$ depend only on $(\pi,\pi',\epsilon)$.

If $n=n'$, then we must account for the possible poles of the $L$-functions in \eqref{eqn:Lfns} when applying \cref{prop:P1}.  Keeping this in mind, we shall prove \eqref{eqn:finalbound3} in three steps.  Each step relies on an application of 
\cref{prop:P1} and allows us to verify \eqref{eqn:finalbound3} for a larger \emph{subgroup} of characters $\chi\in\mathfrak{F}_1$ than before. Let us introduce the notation
\[
\mathfrak{F}_1^{(j)}=\{\chi\in\mathfrak{F}_1:{\chi^*}^j=\mathbbm{1}\},
\]
and note the chain of subgroups $\mathfrak{F}_1^{(1)}\leq\mathfrak{F}_1^{(2)}\leq\mathfrak{F}_1$.

In the first step, we prove \eqref{eqn:finalbound3} for $\chi\in\mathfrak{F}_1^{(1)}$. In the second step, we extend the validity of \eqref{eqn:finalbound3} to $\chi\in\mathfrak{F}_1^{(2)}$. In the third step, we prove \eqref{eqn:finalbound3} in full generality.
For now, let us assume that $L(s,\pi\times(\pi'\otimes\chi))$ is entire. At each application of \cref{prop:P1}, we 
rely on the ``existence of bad zero'' hypothesis that we paraphrase as follows. For the character $\chi\in\mathfrak{F}_1$ under consideration, there exists $(\beta,\eta)\in(1-\epsilon',1)\times\mathfrak{F}_1^{(1)}$ such that 
\[L(\beta,\pi\times(\pi'\otimes\chi\eta))=0.\]
However, we cannot use $(\beta,\eta)$ directly because the implied constants are not allowed to depend on $\chi$. Instead, we use that each step operates with a given subgroup $G\in\{\mathfrak{F}_1^{(1)},\mathfrak{F}_1^{(2)},\mathfrak{F}_1\}$, and by the above hypothesis there exists $(\beta,\psi)\in(1-\epsilon',1)\times G$ such that 
\begin{equation}
\label{eqn:siegelhypothesis}
\text{$L(s,\pi\times(\pi'\otimes\psi))$ is entire and vanishes at $s=\beta$.}
\end{equation}
This \emph{weaker} hypothesis follows from the previous one (upon setting $\psi=\chi\eta$), and plays the role of \eqref{eqn:assumed_zero}. Its advantage lies in the observation that since $G$ is fixed, we can again select an admissible pair $(\beta,\psi)$ solely in terms of the triple $(\pi,\pi',\epsilon)$.

\section{Establishing the key proposition}
\label{sec:proof_prop}

The goal of this section is to prove \Cref{prop:P1}. We begin with a finiteness statement on $\GL_1$-twists, which will help us to control the poles of the $L$-functions in \eqref{eqn:Lfns}.

\begin{lemma}
\label{lem:twists}
Let $\pi\in\mathfrak{F}_n$ and $\pi'\in\mathfrak{F}_{n'}$. There are finitely many $\chi\in\mathfrak{F}_1^*$ such that $L(s,\pi\times(\pi'\otimes\chi))$ has a pole.
\end{lemma}

\begin{proof} Without loss of generality, let $(\pi,\pi')\in\mathfrak{F}_n^*\times\mathfrak{F}_{n'}^*$. Choose $\chi\in\mathfrak{F}_1^*$ such that $L(s,\pi\times(\pi'\otimes\chi))$ has a pole. Then $\pi'\otimes\chi=\tilde\pi$, and $n=n'$. Therefore, $\chi$ can only ramify where $\pi$ or $\pi'$ ramifies. Moreover, $\chi^n$ is uniquely determined by the equation
$\omega_{\pi'}\chi^n=\bar{\omega_\pi}$. Therefore, it suffices to show that if $S$ is a finite set of places containing all archimedean places and
\[
G=\prod_{v\in S}{F_v^\times}^n\prod_{v\not\in S}\cO_{F_v}^\times,
\]
then $F^\times G$ has finite index in $\A_F^\times$. Let $\A_F^1$ be the group of ideles of norm $1$, and let $G^1=\A_F^1\cap G$. We have that
\[
\A_F^\times/(F^\times G)\cong\A_F^1/(F^\times G^1)\cong(\A_F^1/F^\times)/(F^\times G^1/F^\times),
\]
where $G$ is open in $\A_F^\times$, $G^1$ is open in $\A_F^1$, and $F^\times G^1/F^\times$ is open in $\A_F^1/F^\times$.
As $\A_F^1/F^\times$ is compact, the right-hand side is finite, and so is the left-hand side.
\end{proof}

The next lemma prepares the scene for our main residue calculation.

\begin{lemma}
\label{lem:residue}
Let $f_0(s),\dotsc,f_m(s)$ be $m+1$ complex functions that are holomorphic in an open neighborhood of $s_0\in\CC$. If there exists $c\geq 0$ such that $|f_j(s_0)|=c$ for all $j\in\{1,\dotsc,m\}$, then
\[
\mathop{\mathrm{Res}}_{s=s_0}\frac{f_0(s)\dotsb f_m(s)}{(s-s_0)^m}
\]
equals $c$ times a $\CC$-linear combination of monomials of the derivative values $f_j^{(k)}(s_0)$ for $(j,k)\in\{0,\dotsc,m\}\times\{0,\dotsc,m-1\}$. The monomials in the linear combination, and the modulus of each coefficient in the linear combination depend at most on $m$.
\end{lemma}

\begin{proof}
Expanding $f_0(s),\dotsc,f_m(s)$ into Taylor series around $s_0$, we obtain
\[
\mathop{\mathrm{Res}}_{s=s_0}\frac{f_0(s)\dotsb f_m(s)}{(s-s_0)^m} = 
\sum_{\substack{k_0,\dotsc,k_m\geq 0\\ k_0+\dotsb+k_m=m-1}}\prod_{j=0}^m\frac{f_j^{(k_j)}(s_0)}{k_j!}.
\]
Since $k_0+\dotsb+k_m=m-1$, at least one of $k_1,\dotsc,k_{m}$ equals zero. Consequently, the $j$-product above has a factor of $f_j(s_0)$ for some $j\in\{1,\dots,m\}$, and the result follows.
\end{proof}

We apply \cref{lem:residue} to study the residues of an auxiliary $L$-function.  For $(\pi,\pi',\chi)\in\mathfrak{F}_n\times\mathfrak{F}_{n'}\times\mathfrak{F}_1$, we define
\[
\Pi=\pi\boxplus\pi\otimes\chi\boxplus\tilde\pi'\boxplus\tilde\pi'\otimes\bar{\chi}
\qquad\text{and}\qquad
D(s)=L(s,\Pi\times\tilde{\Pi}).
\]
We have the factorization
\begin{equation}
\label{eqn:D_def}
\begin{aligned}
D(s)=&~L(s,\pi\times\tilde\pi)^2 L(s,\pi'\times\tilde\pi')^2 L(s,\pi\times(\pi'\otimes\chi))^2 L(s,\tilde\pi\times(\tilde\pi'\otimes\bar{\chi}))^2\\
&\cdot L(s,\pi\times(\tilde\pi\otimes\chi)) L(s,\pi'\times(\tilde\pi'\otimes\chi)) L(s,\tilde\pi\times\tilde\pi') L(s,\pi\times(\pi'\otimes\chi^2))\\
&\cdot L(s,\pi\times(\tilde\pi\otimes\bar{\chi})) L(s,\pi'\times(\tilde\pi'\otimes\bar{\chi})) L(s,\pi\times\pi') L(s,\tilde\pi\times(\tilde\pi'\otimes\bar{\chi}^2)).
\end{aligned}
\end{equation}
It follows from \cref{lem:nonneg} that $D(s)$ has nonnegative Dirichlet coefficients.  Let us also introduce
\begin{equation}
\label{eqn:Qdef}
Q=(C(\pi)C(\pi'))^{2(n+n')}C(\chi)^{(n+n')^2}.
\end{equation}
The information we need about the residues of $D(s)$ is as follows.
\begin{lemma}
\label{lem:residue_bounds}
Let $(\pi,\pi',\chi)\in\mathfrak{F}_n\times\mathfrak{F}_{n'}\times\mathfrak{F}_1$, $x>1$, $\epsilon\in(0,1)$, and $\beta\in(1-\epsilon/2,1)$. Recall the notations \eqref{eqn:pidecomp}, \eqref{eqn:D_def}, \eqref{eqn:Qdef}, and let $\mathcal{S}$ be the set of poles of $D(s)$. Assume that the $L$-functions in \eqref{eqn:Lfns} are entire.  If $\pi\otimes\chi^*=\pi$ or $\pi'\otimes\chi^*=\pi'$, then assume also that $|t_\chi|>1$. We have that
\begin{equation}
\label{eqn:lemmabound}
\sum_{s_0\in\mathcal{S}}\mathop{\mathrm{Res}}_{s=s_0}D(s)x^{s-\beta}\Gamma(s-\beta)
\ll_{n,n',[F:\Q],\beta,\epsilon}|L(1,\pi\times(\pi'\otimes\chi))|(Qx)^{\epsilon}.	
\end{equation}
\end{lemma}

\begin{proof}
First, assume that $\pi\otimes\chi^*=\pi$ and $\pi'\otimes\chi^*=\pi'$. Then $|t_{\chi}|>1$ by hypothesis, and $\mathcal{S}=\{1,1-it_\chi,1+it_\chi\}$. For each choice of $s_0\in\mathcal{S}$, let $m$ be the order of the pole of $D(s)$ at $s=s_0$. Specifically, $m=4$ for $s_0=1$, and $m=2$ for $s_0=1\pm it_\chi$. Consider the decomposition
\begin{equation}
\label{eqn:decomposition}
(s-s_0)^m D(s)x^{s-\beta}\Gamma(s-\beta)=f_0(s)\dotsb f_m(s),
\end{equation}
where
\begin{itemize}
\item $f_1(s)=f_2(s)=L(s+it_\chi,\pi\times\pi')$ and $f_3(s)=f_4(s)=L(s-it_\chi,\tilde{\pi}\times\tilde{\pi}')$ for $s_0=1$;
\item $f_1(s)=L(s,\tilde\pi\times\tilde\pi')$ and $f_2(s)=L(s+2it_\chi,\pi\times\pi')$ for $s_0=1-it_\chi$;
\item $f_1(s)=L(s,\pi\times\pi')$ and $f_2(s)=L(s-2it_\chi,\tilde{\pi}\times\tilde{\pi}')$ for $s_0=1+it_\chi$.
\end{itemize}
These three cases correspond to the three lines in \eqref{eqn:D_def}, with $f_1(s),\dotsc,f_{m}(s)$ occurring as $m$ factors on the relevant line. Now we apply \cref{lem:residue} in conjunction with \cref{lem:Li1} and \eqref{eqn:BH}. The functions $f_0(s),\dotsc,f_m(s)$ defined above are holomorphic in the open disk $|s-s_0|<1-\beta$. Moreover,
\[
|f_j(s_0)|=|L(1,\pi\times(\pi'\otimes\chi))|,\qquad j\in\{1,\dots,m\}.
\]
We are finished upon noting that, at the point $s_0$, the $k$-th derivative of $s\mapsto x^{s-\beta}$ is bounded by $x^{\epsilon/2}(\log x)^k$, while
\[
\left|\Gamma^{(k)}(s_0-\beta)\right|\leq\int_0^\infty r^{-\beta}|\log r|^k e^{-r}\,dr.
\]

Second, assume that exactly one of $\pi\otimes\chi^*=\pi$ and $\pi'\otimes\chi^*=\pi'$ holds true. Then $|t_{\chi}|>1$ by hypothesis, and $\mathcal{S}=\{1,1-it_\chi,1+it_\chi\}$. For each choice of $s_0\in\mathcal{S}$, let $m$ be the order of the pole of $D(s)$ at $s=s_0$. Specifically, $m=4$ for $s_0=1$, and $m=1$ for $s_0=1\pm it_\chi$. Consider the decomposition \eqref{eqn:decomposition}, where
\begin{itemize}
\item $f_1(s)=f_2(s)=L(s+it_\chi,\pi\times\pi')$ and $f_3(s)=f_4(s)=L(s-it_\chi,\tilde{\pi}\times\tilde{\pi}')$ for $s_0=1$;
\item $f_1(s)=L(s,\tilde\pi\times\tilde\pi')$ for $s_0=1-it_\chi$;
\item $f_1(s)=L(s,\pi\times\pi')$ for $s_0=1+it_\chi$.
\end{itemize}
From here we proceed exactly as in the previous case.

Finally, assume that $\pi\otimes\chi^*\neq\pi$ and $\pi'\otimes\chi^*\neq\pi'$. Then $\mathcal{S}=\{1\}$ by hypothesis. The order of the pole of $D(s)$ at $s=1$ is 4. We consider the decomposition \eqref{eqn:decomposition}, where 
\[
f_1(s)=f_2(s)=L(s,\pi\times(\pi'\otimes\chi))\quad\text{and}\quad f_3(s)=f_4(s)=L(s,\tilde\pi\times(\tilde\pi'\otimes\bar{\chi})).
\]
These four factors occur on the first line of \eqref{eqn:D_def}, and we finish as in the other cases.
\end{proof}

We use \cref{lem:nonneg,lem:Li1,lem:residue_bounds} to prove \cref{prop:P1}.

\begin{proof}[Proof of \cref{prop:P1}]
Recall \eqref{eqn:pidecomp} and \eqref{eqn:Qdef}. If $\pi\otimes\chi^*=\pi$ or $\pi'\otimes\chi^*=\pi'$, then we may assume that $|t_\chi|>1$, because the left-hand side of \eqref{eqn:main_lower} equals $|L(1+it_\chi,\pi\times\pi')|$, a positive continuous function of $t_\chi\in\R$ by Shahidi's nonvanishing result \cite[Theorem~5.2]{Shahidi}, while the right-hand side is less than $1$. Subject to this constraint, let $D(s)$ be as in \eqref{eqn:D_def}, let $\mathcal{S}$ be the set of its poles, and let $x>1$ be a parameter to be chosen later. By the initial assumptions, $D(\beta)=0$. Hence $\mathcal{S}$ is also the set of poles of $D(s)x^{s-\beta}\Gamma(s-\beta)$ in the half-plane $\Re(s)>0$, and we shall use this below.

If $\lambda_{D}(\ka)$ is the $\ka$-th Dirichlet coefficient of $D(s)$, then $\lambda_{D}(\ka)\geq 0$ by \cref{lem:nonneg}. Since $\lambda_{D}(\cO_F)=1$, we have by the residue theorem
\begin{align*}
\frac{1}{e}\leq \sum_{\ka}\frac{\lambda_{D}(\ka)}{\N\ka^{\beta}}e^{-\frac{\N\ka}{x}}&=\frac{1}{2\pi i}\int_{1-i\infty}^{1+i\infty}D(s+\beta)x^s\Gamma(s)\,ds\\
&=\sum_{s_0\in\mathcal{S}}\mathop{\mathrm{Res}}_{s=s_0}D(s)x^{s-\beta}\Gamma(s-\beta)
+\frac{1}{2\pi i}\int_{1/2-i\infty}^{1/2+i\infty}D(s)x^{s-\beta}\Gamma(s-\beta)\,ds.
\end{align*}
We estimate the sum over $\mathcal{S}$ by \cref{lem:residue_bounds} (with $\epsilon$ replaced by $\epsilon/4$), and the last integral by
\cref{lem:Li1} (with $\epsilon$ replaced by $\epsilon/16$) combined with \eqref{eqn:BH} and Stirling's formula. We conclude that
\begin{equation}
\label{eqn:previousbound}
1\ll_{n,n',[F:\Q],\beta,\epsilon} \left(|L(1,\pi\times(\pi'\otimes\chi))|+Qx^{-1/2}\right)(Qx)^{\epsilon/4}.	
\end{equation}
At this point, we choose
\[x = \max\left(1,Q^{2}|L(1,\pi\times(\pi'\otimes\chi))|^{-2}\right).\]
If $x=1$, then \eqref{eqn:main_lower} is trivial. Otherwise, $x=Q^{2}|L(1,\pi\times(\pi'\otimes\chi))|^{-2}>1$, and \eqref{eqn:previousbound} yields
\eqref{eqn:main_lower} after solving for $|L(1,\pi\times(\pi'\otimes\chi))|$:
\[|L(1,\pi\times(\pi'\otimes\chi))|\gg_{n,n',[F:\Q],\beta,\epsilon} Q^{-3\epsilon/(4-2\epsilon)}>Q^{-\epsilon}.\qedhere\]
\end{proof}

\begin{remark}
The proof of \cref{prop:P1} relies crucially on the nonnegativity of the Dirichlet coefficients of $D(s)$. It is instructive to see that in the special case when $F=\Q$ and $\pi=\pi'=\mathbbm{1}$, the auxiliary $L$-function $D(s)$ becomes
\[
\zeta(s)^6 L(s,\chi)^4 L(s,\bar{\chi})^4 L(s,\chi^2) L(s,\bar{\chi}^2).
\]
The nonnegativity of the Dirichlet coefficients of (the logarithm of) this $L$-function is the basis of the standard zero-free region \eqref{eqn:standard_Dirichlet} for Dirichlet $L$-functions. Moreover, if $F=\Q$ and $\chi,\psi\in\mathfrak{F}_1$ are two quadratic characters, then for $\pi=\mathbbm{1}$ and $\pi'=\psi$ the auxiliary $L$-function $D(s)$ becomes
\[
\zeta(s)^4 L(s,\chi)^4 L(s,\psi)^4 L(s,\chi\psi)^4.
\]
The nonnegativity of the Dirichlet coefficients of this $L$-function (or in fact the same without the exponents $4$) is the basis of Siegel's lower bound on $L(1,\chi)$, leading to the zero-free interval \eqref{eqn:Siegel_Dirichlet}. These observations indicate that $D(s)$ is a very natural object.
\end{remark}

\section{Proof of \cref{thm:Siegel} for $\sigma=1$.}
\label{sec:main_case}

The goal of this section is to deduce \eqref{eqn:finalbound3} from \cref{prop:P1}. As in \cref{sec:Strategy}, we shall assume without loss of generality that $\epsilon\in(0,1/2)$ and use the notation \eqref{eqn:epsprime}.

\subsection{An initial reduction}
\label{subsec:inital_reduction}

We shall assume that $L(s,\pi\times(\pi'\otimes\chi))$ is holomorphic in the open disk $|s-1|<1$, for otherwise \eqref{eqn:finalbound3} is clear. Moreover, if $L(s,\pi\times(\pi'\otimes\chi))$ has no zero in the half-plane $\Re(s)>1-\epsilon'$, then standard methods produce a bound much stronger than \eqref{eqn:finalbound3}. Indeed, assume this zero-free region, and let $\sigma\in[1,2]$. Let us work with a parameter $x\geq 2$ (to be chosen later) and the Mellin transform pair
\[\phi(r)=\max(1-r,0),\qquad \hat\phi(w)=1/(w^2+w).\]
Proceeding as in the proof of \cite[Proposition~5.16]{IK}, but shifting the contour only to $\Re(w)=-\epsilon'/2$, we infer that
\[
-\frac{L'}{L}(\sigma,\pi\times(\pi'\otimes\chi))=
\sum_{\ka}\frac{\Lambda_{\pi\times(\pi'\otimes\chi)}(\ka)}{\N\ka^{\sigma}}\phi\left(\frac{\N\ka}{x}\right)\\
+O_{\pi,\pi',\epsilon}\left(x^{-\epsilon'/2}\log C(\chi)\right).
\]

Since $C(\chi)\geq 3$ and $0<\epsilon'<1/64$, the choice $x=(\log C(\chi))^{2/\epsilon'}$ satisfies $x\geq 3$ and $x^{-\epsilon'/2}\log C(\chi)=1$. Integrating the resulting approximation from $\sigma=1$ to $\sigma=2$, and applying the triangle inequality, we infer
\begin{align*}
|\log L(1,\pi\times(\pi'\otimes\chi))|&\leq
\sum_{2\leq \N\ka\leq x}\frac{|\Lambda_{\pi\times(\pi'\otimes\chi)}(\ka)|}{\N\ka\log\N\ka}+O_{\pi,\pi',\epsilon}(1)\\
&\ll_{\pi,\pi',\epsilon} \sum_{k=1}^{\lfloor \log x\rfloor}\frac{1}{ke^k}\sum_{e^{k-1}<\N\ka\leq e^k}|\Lambda_{\pi\times(\pi'\otimes\chi)}(\ka)|+1.
\end{align*}
By \cref{cor:BT}, we conclude the bound
\begin{align*}
|\log|L(1,\pi\times(\pi'\otimes\chi))||
&\leq |\log L(1,\pi\times(\pi'\otimes\chi))|\\
&\ll_{\pi,\pi',\epsilon}\log\log x\\
&\ll_{\pi,\pi',\epsilon}\log(1+\log\log C(\chi)).
\end{align*}
Consequently, there exists a constant $\Cl[abcon]{lower_ref}=\Cr{lower_ref}(\pi,\pi',\epsilon)>0$ such that
\[(1+\log\log C(\chi))^{-\Cr{lower_ref}}\leq |L(1,\pi\times(\pi'\otimes\chi))|\leq(1+\log\log C(\chi))^{\Cr{lower_ref}},\]
which is much stronger than \eqref{eqn:finalbound3}.

\subsection{The case of $\chi^*$ trivial} In this subsection, we prove \eqref{eqn:finalbound3} for all $\chi$ lying in the subgroup 
$G=\mathfrak{F}_1^{(1)}$. Let us write $\chi=|\cdot|^{it}\in G$ with $t\in\R$. If $L(s,\pi\times\pi')$ has a pole, then there exists $u\in\R$ such that $\pi'=\tilde\pi\otimes|\cdot|^{iu}$, hence \eqref{eqn:finalbound3} holds in the stronger form
\[
|L(1+it,\pi\times\pi')|\gg_{\pi,\pi'}1/\log(|t|+3)
\]
by appealing to the zero-free region in \cite[Theorem~2.1]{HumphriesThorner}. Therefore, we shall assume that $L(s,\pi\times\pi')$ is entire. Furthermore, as explained in \cref{sec:Strategy}, we can assume that \eqref{eqn:siegelhypothesis} holds for some $(\beta,\psi)\in(1-\epsilon',1)\times G$ depending only on $(\pi,\pi',\epsilon)$.

Consider the automorphic representations
\begin{equation}
\label{eqn:changeofvariable}
\pi''=\pi'\otimes\psi\in\mathfrak{F}_{n'}\qquad\text{and}\qquad\chi'=\bar\psi\chi\in G.
\end{equation}
It follows that
\[L(s,\pi\times(\pi''\otimes\chi'))=L(s,\pi\times(\pi'\otimes\chi))\qquad\text{and}\qquad C(\chi')\asymp_{\pi,\pi',\epsilon}C(\chi),\]
and \eqref{eqn:finalbound3} is equivalent to
\begin{equation}
\label{eqn:finalbound4}
|L(1,\pi\times(\pi''\otimes\chi'))|\gg_{\pi,\pi'',\epsilon}C(\chi')^{-\epsilon}.
\end{equation}
But \eqref{eqn:finalbound4} follows readily from \cref{prop:P1}, since $L(\beta,\pi\times\pi'')=0$ and the following $L$-functions are entire:
\begin{equation}
\label{eqn:Lfns2}
L(s,\pi\times\pi''),\qquad L(s,\pi\times(\pi''\otimes\chi')),\qquad L(s,\pi\times(\pi''\otimes\chi'^2)).
\end{equation}
Indeed, these $L$-functions are shifts of $L(s,\pi\times\pi')$, which is entire by assumption.

We have shown that \eqref{eqn:finalbound3} holds for all $\chi\in\mathfrak{F}_1^{(1)}$. Consequently, \eqref{eqn:finalbound3} also holds for all $\chi$ in any \emph{fixed coset} of $\mathfrak{F}_1^{(1)}$ within $\mathfrak{F}_1$. We shall use this principle in the next subsection.

\subsection{The case of $\chi^*$ quadratic} In this subsection, we prove \eqref{eqn:finalbound3} for all $\chi$ lying in the subgroup 
$G=\mathfrak{F}_1^{(2)}$. If $L(s,\pi\times(\pi'\otimes\chi))$ has a pole, then by \cref{lem:twists}, $\chi$ lies in finitely many cosets of $\mathfrak{F}_1^{(1)}$ (depending only on $(\pi,\pi')$), hence \eqref{eqn:finalbound3} holds by the concluding remark of the previous subsection. Therefore, we shall assume that $L(s,\pi\times(\pi'\otimes\chi))$ is entire.

As before, we also assume that \eqref{eqn:siegelhypothesis} holds for some $(\beta,\psi)\in(1-\epsilon',1)\times G$ depending only on $(\pi,\pi',\epsilon)$, and we need to prove \eqref{eqn:finalbound4} with the notation \eqref{eqn:changeofvariable}. But \eqref{eqn:finalbound4} follows readily from \cref{prop:P1}, upon noting that $L(\beta,\pi\times\pi'')=0$ and the $L$-functions in \eqref{eqn:Lfns2} are entire. Indeed, the first two $L$-functions in \eqref{eqn:Lfns2} are entire by assumption, while the third $L$-function is a shift of the first one due to $\chi\in G$.

We have shown that \eqref{eqn:finalbound3} holds for all $\chi\in\mathfrak{F}_1^{(2)}$. Consequently, \eqref{eqn:finalbound3} also holds for all $\chi$ in any \emph{fixed coset} of $\mathfrak{F}_1^{(2)}$ within $\mathfrak{F}_1$. We shall use this principle in the next subsection.

\subsection{The general case} In this subsection, we prove \eqref{eqn:finalbound3} in general. As before, we assume that $L(s,\pi\times(\pi'\otimes\chi))$ is entire, and \eqref{eqn:siegelhypothesis} holds for some $(\beta,\psi)\in(1-\epsilon',1)\times \mathfrak{F}_1$ depending only on $(\pi,\pi',\epsilon)$. We need to prove \eqref{eqn:finalbound4} with the notation \eqref{eqn:changeofvariable}. 

If $L(s,\pi\times(\pi''\otimes\chi'^2))$ has a pole, then by \cref{lem:twists}, $\chi'$ lies in finitely many cosets of $\mathfrak{F}_1^{(2)}$ depending only on $(\pi,\pi'')$, hence \eqref{eqn:finalbound4} holds by the concluding remark of the previous subsection. Therefore, we shall assume that $L(s,\pi\times(\pi''\otimes\chi'^2))$ is entire. But then \eqref{eqn:finalbound4} follows readily from \cref{prop:P1}, because $L(\beta,\pi\times\pi'')=0$ and the $L$-functions in \eqref{eqn:Lfns2} are entire by assumption.

\section{Finishing the proof of \cref{thm:Siegel}}
\label{sec:finish}

Now that we have proved \eqref{eqn:finalbound3} for all $(\pi,\pi',\chi)\in\mathfrak{F}_n\times\mathfrak{F}_{n'}\times\mathfrak{F}_1$ and $\epsilon>0$, we can finish the proof of \cref{thm:Siegel}. As before, by continuity and nonvanishing arguments, we can assume that $L(s,\pi\times(\pi'\otimes\chi))$ is holomorphic in the open disk $|s-1|<1$.

\subsection{The case of $\sigma<1$}\label{subsub}
We begin with the straightforward bound
\[
|L(1,\pi\times(\pi'\otimes\chi))-L(\sigma,\pi\times(\pi'\otimes\chi))|\leq(1-\sigma)\sup_{\kappa\in[\sigma,1]}|L'(\kappa,\pi\times(\pi'\otimes\chi))|.
\]
For $|L(1,\pi\times(\pi'\otimes\chi))|$, we have the lower bound \eqref{eqn:finalbound3}. On the other hand, by \cref{lem:Li1} and \eqref{eqn:BH}, there exists a constant $\Cl[abcon]{ZFRN2}=\Cr{ZFRN2}(\pi,\pi',\epsilon)>0$ such that
\[
|L'(\kappa,\pi\times(\pi'\times\chi))|\leq\Cr{ZFRN2}C(\chi)^{\epsilon/2},\qquad\kappa\geq 1-\frac{\epsilon}{2nn'}.
\]
Define
\[
\Cl[abcon]{ZFRN3}=\Cr{ZFRN3}(\pi,\pi',\epsilon)=\min\left(\frac{\Cr{ZFR3}}{1+\Cr{ZFRN2}},\frac{\epsilon}{2nn'}\right).
\]
If $1-\Cr{ZFRN3}C(\chi)^{-\epsilon}<\sigma<1$, then
\begin{align*}
|L(\sigma,\pi\times(\pi'\otimes\chi))|
&\geq |L(1,\pi\times(\pi'\otimes\chi))|-|L(1,\pi\times(\pi'\otimes\chi))-L(\sigma,\pi\times(\pi'\otimes\chi))|\\
&\geq |L(1,\pi\times(\pi'\otimes\chi))|-\Cr{ZFRN3}C(\chi)^{-\epsilon}\cdot\Cr{ZFRN2}C(\chi)^{\epsilon/2}\\
&\geq (\Cr{ZFR3}- \Cr{ZFRN2}\Cr{ZFRN3})C(\chi)^{-\epsilon/2}\\
&\geq \Cr{ZFRN3} C(\chi)^{-\epsilon/2}.
\end{align*}

\subsection{The case of $\sigma>1$}

It follows from \eqref{eqn:GRC1} and \eqref{eqn:GRC2} that
\begin{equation}
\label{eqn:edge2}
L(\sigma,\pi\times(\pi'\otimes\chi))\asymp_{n,n'}1,\qquad\sigma\geq 3.
\end{equation}
To prove \eqref{eqn:finalbound2} in the strip $1<\sigma<3$, we interpolate between \eqref{eqn:finalbound3} and \eqref{eqn:edge2} by applying Phragm{\'e}n--Lindel{\"o}f principle to $1/L(s,\pi\times\pi')$.

\section{Proof of \cref{thm:PNTAP}}
\label{sec:PNTAP}

Let $x\geq 2$, $A>0$, and $y=x(\log x)^{-A}$. Consider the Mellin transform pair
\begin{equation}
\label{eqn:hatphi}
\begin{aligned}
\phi(r)&=\mathbf{1}_{(0,x]}(r)+\mathbf{1}_{(x,x+y]}(r)\frac{x+y-r}{y},\qquad r>0;\\
\hat\phi(s)&=\int_{0}^{\infty}\phi(r)r^{s-1}\,dr=\frac{(x+y)^{s+1}-x^{s+1}}{y(s^2+s)},\qquad\Re(s)>0.
\end{aligned}
\end{equation}
Clearly,
\begin{align*}
\sum_{\substack{\N\ka\leq x \\ \ka\in\mathcal{C}}}\Lambda_{\pi\times\pi'}(\ka) = \sum_{\substack{ \ka\in\mathcal{C}}}\Lambda_{\pi\times\pi'}(\ka)\phi(\N\ka)+O\left(\sum_{x<\N\ka\leq x+y}|\Lambda_{\pi\times\pi'}(\ka)|\right).
\end{align*}
By \cref{cor:BT}, the error term above is $O_{\pi,\pi'}(y)$.

Let $\widehat{\mathrm{Cl}(\kq)}$ be the group of characters of $\mathrm{Cl}(\kq)$. By character orthogonality, we have that
\[
\sum_{\substack{\ka\in\mathcal{C}}}\Lambda_{\pi\times\pi'}(\ka)\phi(\N\ka)=\frac{1}{|\mathrm{Cl}(\kq)|}\sum_{\psi\in\widehat{\mathrm{Cl}(\kq)}}\bar{\psi}(\mathcal{C})\mathcal\sum_{\ka}\Lambda_{\pi\times\pi'}(\ka)\psi(\ka)\phi(\N\ka).
\]
Using \eqref{eqn:GRC1} and \eqref{eqn:GRC2}, we find that if $\chi\in\mathscr{P}(\kq)$ is the primitive ray class character that induces $\psi$, then
\begin{align*}
&\sum_{\ka}\Lambda_{\pi\times\pi'}(\ka)\psi(\ka)\phi(\N\ka)-\sum_{\ka}\Lambda_{\pi\times(\pi'\otimes\chi)}(\ka)\phi(\N\ka)\\
&\ll nn'\sum_{\kp\mid\kq\kq_{\pi}\kq_{\pi'}}\sum_{\ell\leq \frac{\log (2x)}{\log\N\kp}}\N\kp^{\ell(\theta_n+\theta_{n'})}\log\N\kp\\
&\ll_{\pi,\pi'} x^{1-\frac{1}{n^2+1}-\frac{1}{(n')^2+1}}(\log x)(\log\N\kq).
\end{align*}
In light of the hypothesis $\N\kq\leq (\log x)^A$ and our choice of $y$, it follows that
\[
\sum_{\substack{\N\ka\leq x \\ \ka\in\mathcal{C}}}\Lambda_{\pi\times\pi'}(\ka)=\frac{1}{|\mathrm{Cl}(\kq)|}\sum_{\chi\in\mathscr{P}(\kq)}\overline{\chi}(\mathcal{C})\sum_{\ka}\Lambda_{\pi\times(\pi'\otimes\chi)}(\ka)\phi(\N\ka)+O_{\pi,\pi',A}(y).
\]

Recall the definition of $\mathcal{E}_{\pi\times\pi'}(x;\kq,\mathcal{C})$ from \eqref{eqn:Edef}. Applying Mellin inversion, we obtain
\begin{multline*}
\mathcal{E}_{\pi\times\pi'}(x;\kq,\mathcal{C})=\frac{1}{|\mathrm{Cl}(\kq)|}\sum_{\chi\in\mathscr{P}(\kq)}\frac{\overline{\chi}(\mathcal{C})}{2\pi i}\int_{3-i\infty}^{3+i\infty}-\frac{L'}{L}(s,\pi\times(\pi'\otimes\chi))\,\hat\phi(s)\,ds\\
-\frac{1}{|\mathrm{Cl}(\kq)|}\sum_{\chi\in\mathscr{P}(\kq)}\overline{\chi}(\mathcal{C})\mathcal{M}_{\pi\times(\pi'\otimes\chi)}(x)+O_{\pi,\pi',A}(y).
\end{multline*}
Define $\epsilon=1/(3A+2)$. By \eqref{eqn:ZFR}, \cref{lem:Li1}, \eqref{eqn:BH}, and the bound $C(\chi)\ll_{[F:\Q]}\N\kq$, there exists an ineffective constant $\Cl[abcon]{C7777}=\Cr{C7777}(\pi,\pi',\epsilon)>0$ such that on the piecewise smooth parametric curve
\[
\mathscr{C}(t)=1-\Cr{C7777}(\N\kq(|t|+1))^{-\epsilon}+it,\qquad t\in\R,
\]
there holds
\begin{equation}
\label{eqn:ZFR_twisted2}
\frac{L'}{L}(\mathscr{C}(t),\pi\times(\pi'\otimes\chi))\ll_{\pi,\pi',\epsilon}(\N\kq(|t|+1))^{2\epsilon}.
\end{equation}
Note that $\mathscr{C}'(t)\ll_{\pi,\pi',\epsilon} 1$ for $t\neq 0$. We deform the line of integration to $\mathscr{C}$. 

If there exists $u\in\R$ such that $\pi'\otimes\chi=\tilde{\pi}\otimes|\cdot|^{iu}$, then the integrand has a pole at $s=1-iu$ with residue
\[
\hat\phi(1-iu)=\int_0^x r^{-iu}\,dr + \int_x^{x+y}\phi(r)r^{-iu}\,dr=\frac{x^{1-iu}}{1-iu}+O(y)=\mathcal{M}_{\pi\times(\pi'\otimes\chi)}(x)+O(y).
\]
If no such $u$ exists, then $L(s,\pi\times(\pi'\otimes\chi))$ is entire, and $\mathcal{M}_{\pi\times(\pi'\otimes\chi)}(x)=0$. Therefore, by the residue theorem,
\begin{align*}
\mathcal{E}_{\pi\times\pi'}(x;\kq,\mathcal{C})&=\frac{1}{|\mathrm{Cl}(\kq)|}\sum_{\chi\in\mathscr{P}(\kq)}\frac{\overline{\chi}(\mathcal{C})}{2\pi i}\int_{\mathscr{C}}-\frac{L'}{L}(s,\pi\times(\pi'\otimes\chi))\,\hat\phi(s)\,ds+O_{\pi,\pi',A}(y)\\
&\ll_{\pi,\pi',A}\max_{\chi\in\mathscr{P}(\kq)}\left|\int_{\mathscr{C}}-\frac{L'}{L}(s,\pi\times(\pi'\otimes\chi))\,\hat\phi(s)\,ds\right|+y.
\end{align*}
Therefore, we bound the integrand via \eqref{eqn:hatphi} and \eqref{eqn:ZFR_twisted2}:
\[
\mathcal{E}_{\pi\times\pi'}(x;\kq,\mathcal{C})
\ll_{\pi,\pi',A}\frac{x^2}{y}\N\kq^{2\epsilon}\int_{-\infty}^{\infty}x^{-\Cr{C7777}(\N\kq(|t|+1))^{-\epsilon}}(|t|+1)^{2\epsilon-2}\,dt+y.
\]
Define $r=(\N\kq(|t|+1))^{-\epsilon}\log x$. We express the last integral in terms of the new variable $r$ and estimate it in a straightforward fashion, using that $\epsilon\in(0,1/2$):
\[
\mathcal{E}_{\pi\times\pi'}(x;\kq,\mathcal{C})
\ll_{\pi,\pi',A}\frac{x^2}{y}(\log x)^{2-1/\epsilon}\N\kq \int_0^{\N\kq^{-\epsilon}\log x}e^{-\Cr{C7777} r}r^{-3+1/\epsilon}\,dr+y.
\]
Our choices of $y$ and $\epsilon$ and our range of $\N\kq$ ensure that $\mathcal{E}_{\pi\times\pi'}(x;\kq,\mathcal{C})
\ll_{\pi,\pi',A} y$, as desired.

\section{Proof of \cref{thm:sym}}
\label{sec:sym}

We begin with a useful corollary of \cref{thm:Siegel} and \cref{lem:Li1}.

\begin{corollary}
\label{cor:tightness}
Let $(\pi,\pi',\chi)\in\mathfrak{F}_{n}\times\mathfrak{F}_{n'}\times\mathfrak{F}_1$. Let $\mathcal{L}(s,\pi\times\pi')$
be as in \eqref{cL2}. For all $\epsilon>0$, there exists an ineffective constant $\Cl[abcon]{Li_cor_1}=\Cr{Li_cor_1}(\pi,\pi',\epsilon)>0$ and an effective constant $\Cl[abcon]{Li_cor_2}=\Cr{Li_cor_2}(\pi,\pi',\epsilon)>0$ such that if $\sigma\geq 1-\Cr{Li_cor_1}C(\chi)^{-\epsilon}$, then
\[
\Cr{Li_cor_1}C(\chi)^{-\epsilon}\leq |\mathcal{L}(\sigma,\pi\times(\pi'\otimes\chi))|\leq \Cr{Li_cor_2}C(\chi)^{\epsilon}.
\]
\end{corollary}
\begin{proof} Assume that $\Cr{Li_cor_1}\leq\min(\Cr{ZFR2}/3,\epsilon/(nn'))$, where $\Cr{ZFR2}$ is the constant from \cref{thm:Siegel}. The upper bound follows (with suitable $\Cr{Li_cor_2}$) from \cref{lem:Li1} and \eqref{eqn:BH}, hence we focus on the lower bound. If $L(s,\pi\times(\pi'\otimes\chi))$ is entire or $|t_\pi+t_{\pi'}+t_\chi|>1$, then the lower bound follows from \cref{thm:Siegel}.  The case $\sigma>3$ is covered by \eqref{eqn:edge2}. Finally, assume that $L(s,\pi\times(\pi'\otimes\chi))$ has a pole, $|t_\pi+t_{\pi'}+t_\chi|\leq 1$, and $\sigma\leq 3$. Then $\pi'^*\otimes\chi^*=\tilde\pi^*$, and the (finite) values
\[\mathcal{L}(\sigma,\pi\times(\pi'\otimes\chi))=\mathcal{L}(\sigma+i(t_\pi+t_{\pi'}+t_\chi),\pi^*\times\tilde\pi^*)\]
are nonzero by \cref{thm:Siegel}. As $\sigma+i(t_\pi+t_{\pi'}+t_\chi)$ varies in a compact set, the corresponding absolute values have a positive minimum by continuity, and the result follows (with suitable $\Cr{Li_cor_1}$).
\end{proof}

Let $(\pi,\chi)\in\mathfrak{F}_2\times\mathfrak{F}_1$. If $n\in\{1,2,3,4\}$, then $\Sym^n(\pi)\otimes\chi\in\mathfrak{A}_{n+1}$ by \cite{GJ,Kim,KimShahidi}. It follows that there exists a least integer $r_{n,\pi,\chi}\geq 0$ such that $\mathcal{L}(s,\pi,\Sym^n\otimes\chi)$ is entire. Moreover, $L(s,\Sym^n(\pi)\otimes\chi)$ factors as a product of $\chi$-twisted $L$-functions of cuspidal automorphic representations with unitary central characters. Applying \cref{cor:tightness} to each factor, we obtain the conclusions of \cref{thm:sym}.

Now, assume that $n\in\{5,6,7,8\}$. Let $\ell,m\in\{1,2,3,4\}$, and define $\Sym^1(\pi)=\pi$ and $\Sym^0(\pi)=\mathbbm{1}$. The Clebsch--Gordan identities imply that if $\Re(s)$ is suitably large, then
\[
L(s,\pi,\Sym^{\ell}(\pi)\times(\Sym^m(\pi)\otimes\chi))=\prod_{k=0}^{\min(\ell,m)}L(s,\pi,\Sym^{\ell+m-2k}\otimes\chi\omega_{\pi}^k).
\]
It follows that
\[L(s,\pi,\Sym^n\otimes\chi)=
\frac{L(s,\Sym^4(\pi)\times(\Sym^{n-4}(\pi)\otimes\chi))}{L(s,\Sym^3(\pi)\times(\Sym^{n-5}(\pi)\otimes\chi\omega_{\pi}))}.\]
The symmetric powers of $\pi$ on the right-hand side are isobaric automorphic representations. Hence, there exist nonempty sets $A(\pi,n)$ and $B(\pi,n)$ of ordered pairs of cuspidal automorphic representations with unitary central characters such that
\[
L(s,\pi,\Sym^n\otimes\chi)=
\frac{\prod_{(\rho,\rho')\in A(\pi,n)} L(s,\rho\times(\rho'\otimes\chi))}
{\prod_{(\rho,\rho')\in B(\pi,n)} L(s,\rho\times(\rho'\otimes\chi))}.
\]
In particular, the left-hand side is meromorphic on $\mathbb{C}$, and by the work of Kim--Shahidi~\cite{KimShahidi2}, its only potential pole on the line $\Re(s)=1$ is at $s=1-i(nt_\pi+t_\chi)$. Recalling the notation in \eqref{cL1} and \eqref{cL2}, we infer that
\[\cL(s,\pi,\Sym^n\otimes\chi)=
\frac{\prod_{(\rho,\rho')\in A(\pi,n)} \cL(s,\rho\times(\rho'\otimes\chi))}
{\prod_{(\rho,\rho')\in B(\pi,n)} \cL(s,\rho\times(\rho'\otimes\chi))}.\]
Applying \cref{cor:tightness} to each pair $(\rho,\rho')\in A(\pi,n)\cup B(\pi,n)$, we conclude \cref{thm:sym}.

\bibliographystyle{abbrv}
\bibliography{HarcosThornerZFR}

\begin{thebibliography}{10}

\bibitem{Banks}
W.~D. Banks.
\newblock Twisted symmetric-square {$L$}-functions and the nonexistence of
  {S}iegel zeros on {${\rm GL}(3)$}.
\newblock {\em Duke Math. J.}, 87(2):343--353, 1997.

\bibitem{Brumley}
F.~Brumley.
\newblock Effective multiplicity one on $\mathrm{GL}_{N}$ and narrow zero-free
  regions for {R}ankin--{S}elberg {$L$}-functions.
\newblock {\em Amer. J. Math.}, 128(6):1455--1474, 2006.

\bibitem{GJ}
S.~Gelbart and H.~Jacquet.
\newblock A relation between automorphic representations of {${\rm GL}(2)$} and
  {${\rm GL}(3)$}.
\newblock {\em Ann. Sci. \'{E}cole Norm. Sup. (4)}, 11(4):471--542, 1978.

\bibitem{GelbartLapid}
S.~S. Gelbart and E.~M. Lapid.
\newblock Lower bounds for {$L$}-functions at the edge of the critical strip.
\newblock {\em Amer. J. Math.}, 128(3):619--638, 2006.

\bibitem{GoldfeldLi}
D.~Goldfeld and X.~Li.
\newblock A standard zero free region for {R}ankin--{S}elberg {$L$}-functions.
\newblock {\em Int. Math. Res. Not. IMRN}, (22):7067--7136, 2018.

\bibitem{Harcos}
G.~Harcos.
\newblock Uniform approximate functional equation for principal
  {$L$}-functions.
\newblock {\em Int. Math. Res. Not.}, (18):923--932, 2002.

\bibitem{HoffsteinLockhart}
J.~Hoffstein and P.~Lockhart.
\newblock Coefficients of {M}aass forms and the {S}iegel zero.
\newblock {\em Ann. of Math. (2)}, 140(1):161--181, 1994.
\newblock With an appendix by D. Goldfeld, H. and D. Lieman.

\bibitem{HoffsteinRamakrishnan}
J.~Hoffstein and D.~Ramakrishnan.
\newblock Siegel zeros and cusp forms.
\newblock {\em Int. Math. Res. Not.}, (6):279--308, 1995.

\bibitem{Humphries}
P.~Humphries and F.~Brumley.
\newblock Standard zero-free regions for {R}ankin--{S}elberg {$L$}-functions
  via sieve theory.
\newblock {\em Math. Z.}, 292(3-4):1105--1122, 2019.

\bibitem{HumphriesThorner}
P.~Humphries and J.~Thorner.
\newblock Towards a $\mathrm{GL}_n$ variant of the {H}oheisel phenomenon.
\newblock {\em Trans. Amer. Math. Soc.}, 375(3):1801--1824, 2022.

\bibitem{HumphriesThorner2}
P.~Humphries and J.~Thorner.
\newblock Zeros of {R}ankin-{S}elberg {$L$}-functions in families.
\newblock {\em Compos. Math.}, 160(5):1041--1072, 2024.

\bibitem{IM}
Y.~Ichihara and K.~Matsumoto.
\newblock On the {S}iegel--{T}atuzawa theorem for a class of {$L$}-functions.
\newblock {\em Kyushu J. Math.}, 62(1):201--215, 2008.

\bibitem{IK}
H.~Iwaniec and E.~Kowalski.
\newblock {\em Analytic number theory}, volume~53 of {\em American Mathematical
  Society Colloquium Publications}.
\newblock American Mathematical Society, Providence, RI, 2004.

\bibitem{JPSS}
H.~Jacquet, I.~I. Piatetskii-Shapiro, and J.~A. Shalika.
\newblock Rankin--{S}elberg convolutions.
\newblock {\em Amer. J. Math.}, 105(2):367--464, 1983.

\bibitem{JacquetShalika}
H.~Jacquet and J.~A. Shalika.
\newblock A non-vanishing theorem for zeta functions of {${\rm GL}_{n}$}.
\newblock {\em Invent. Math.}, 38(1):1--16, 1976/77.

\bibitem{JS1}
H.~Jacquet and J.~A. Shalika.
\newblock On {E}uler products and the classification of automorphic
  representations. {I}.
\newblock {\em Amer. J. Math.}, 103(3):499--558, 1981.

\bibitem{JS2}
H.~Jacquet and J.~A. Shalika.
\newblock On {E}uler products and the classification of automorphic
  representations. {II}.
\newblock {\em Amer. J. Math.}, 103(4):777--815, 1981.

\bibitem{JLTW}
Y.~Jiang, G.~L\"{u}, J.~Thorner, and Z.~Wang.
\newblock A {B}ombieri--{V}inogradov theorem for higher-rank groups.
\newblock {\em Int. Math. Res. Not. IMRN}, (1):482--535, 2023.

\bibitem{JLW}
Y.~Jiang, G.~L\"{u}, and Z.~Wang.
\newblock Exponential sums with multiplicative coefficients without the
  {R}amanujan conjecture.
\newblock {\em Math. Ann.}, 379(1-2):589--632, 2021.

\bibitem{Kim}
H.~H. Kim.
\newblock Functoriality for the exterior square of {${\rm GL}_4$} and the
  symmetric fourth of {${\rm GL}_2$}.
\newblock {\em J. Amer. Math. Soc.}, 16(1):139--183, 2003.
\newblock With appendix 1 by D. Ramakrishnan and appendix 2 by H. H. Kim and P.
  Sarnak.

\bibitem{KimShahidi2}
H.~H. Kim and F.~Shahidi.
\newblock Cuspidality of symmetric powers with applications.
\newblock {\em Duke Math. J.}, 112(1):177--197, 2002.

\bibitem{KimShahidi}
H.~H. Kim and F.~Shahidi.
\newblock Functorial products for $\mathrm{GL}_2\times\mathrm{GL}_3$ and the
  symmetric cube for $\mathrm{GL}_2$.
\newblock {\em Ann. of Math. (2)}, 155(3):837--893, 2002.
\newblock With an appendix by C. J. Bushnell and G. Henniart.

\bibitem{Lapid}
E.~Lapid.
\newblock On the {H}arish-{C}handra {S}chwartz space of {$G(F)\backslash
  G(\mathbb{A})$}.
\newblock In {\em Automorphic representations and {$L$}-functions}, volume~22
  of {\em Tata Inst. Fundam. Res. Stud. Math.}, pages 335--377. Tata Inst.
  Fund. Res., Mumbai, 2013.
\newblock With an appendix by F. Brumley.

\bibitem{Li}
X.~Li.
\newblock Upper bounds on {$L$}-functions at the edge of the critical strip.
\newblock {\em Int. Math. Res. Not. IMRN}, (4):727--755, 2010.

\bibitem{Luo}
W.~Luo.
\newblock Non-existence of {S}iegel zeros for cuspidal functorial products on
  {$GL(2) \times GL(3)$}.
\newblock {\em Proc. Amer. Math. Soc.}, 151(5):1915--1919, 2023.

\bibitem{LRS}
W.~Luo, Z.~Rudnick, and P.~Sarnak.
\newblock On the generalized {R}amanujan conjecture for {${\rm GL}(n)$}.
\newblock In {\em Automorphic forms, automorphic representations, and
  arithmetic ({F}ort {W}orth, {TX}, 1996)}, volume~66 of {\em Proc. Sympos.
  Pure Math.}, pages 301--310. Amer. Math. Soc., Providence, RI, 1999.

\bibitem{Molteni}
G.~Molteni.
\newblock Upper and lower bounds at {$s=1$} for certain {D}irichlet series with
  {E}uler product.
\newblock {\em Duke Math. J.}, 111(1):133--158, 2002.

\bibitem{Moreno}
C.~J. Moreno.
\newblock Analytic proof of the strong multiplicity one theorem.
\newblock {\em Amer. J. Math.}, 107(1):163--206, 1985.

\bibitem{MullerSpeh}
W.~M\"{u}ller and B.~Speh.
\newblock Absolute convergence of the spectral side of the {A}rthur trace
  formula for $\mathrm{GL}_n$.
\newblock {\em Geom. Funct. Anal.}, 14(1):58--93, 2004.
\newblock With an appendix by E. M. Lapid.

\bibitem{NewtonThorne}
J.~Newton and J.~A. Thorne.
\newblock Symmetric power functoriality for holomorphic modular forms.
\newblock {\em Publ. Math. Inst. Hautes \'{E}tudes Sci.}, 134:1--116, 2021.

\bibitem{NewtonThorne2}
J.~Newton and J.~A. Thorne.
\newblock Symmetric power functoriality for holomorphic modular forms, {II}.
\newblock {\em Publ. Math. Inst. Hautes \'{E}tudes Sci.}, 134:117--152, 2021.

\bibitem{NewtonThorne3}
J.~{Newton} and J.~A. {Thorne}.
\newblock Symmetric power functoriality for {H}ilbert modular forms.
\newblock {\em arXiv e-prints}, Dec. 2022.
\newblock arXiv:2212.03595.

\bibitem{RamakrishnanWang}
D.~Ramakrishnan and S.~Wang.
\newblock On the exceptional zeros of {R}ankin--{S}elberg {$L$}-functions.
\newblock {\em Compositio Math.}, 135(2):211--244, 2003.

\bibitem{Sarnak}
P.~Sarnak.
\newblock Nonvanishing of {$L$}-functions on $\re(s)=1$.
\newblock In {\em Contributions to automorphic forms, geometry, and number
  theory}, pages 719--732. Johns Hopkins Univ. Press, Baltimore, MD, 2004.

\bibitem{Shahidi}
F.~Shahidi.
\newblock On certain {$L$}-functions.
\newblock {\em Amer. J. Math.}, 103(2):297--355, 1981.

\bibitem{Siegel}
C.~L. Siegel.
\newblock {\"U}ber die {C}lassenzahl quadratischer {Z}ahlk{\"o}rper.
\newblock {\em Acta. Arith.}, 1(1):83--86, 1935.

\bibitem{SoundararajanThorner}
K.~Soundararajan and J.~Thorner.
\newblock Weak subconvexity without a {R}amanujan hypothesis.
\newblock {\em Duke Math. J.}, 168(7):1231--1268, 2019.
\newblock With an appendix by F. Brumley.

\bibitem{Tatuzawa}
T.~Tatuzawa.
\newblock On a theorem of {S}iegel.
\newblock {\em Jpn. J. Math.}, 21:163--178 (1952), 1951.

\bibitem{Walfisz}
A.~Walfisz.
\newblock Zur additiven {Z}ahlentheorie. {II}.
\newblock {\em Math. Z.}, 40(1):592--607, 1936.

\bibitem{Zhang}
Q.~Zhang.
\newblock Lower bounds for {R}ankin-{S}elberg {$L$}-functions on the edge of
  the critical strip.
\newblock {\em Acta Arith.}, 211(2):161--171, 2023.

\end{thebibliography}

\end{document}